\documentclass[12pt,a4paper,twoside]{article}
\usepackage[hmarginratio=1:1,bmargin=1.3in]{geometry}
\usepackage[frenchb]{babel}
\usepackage[utf8]{inputenc}
\usepackage[OT2,T1]{fontenc}
\usepackage{amsmath, amsthm}
\usepackage{amsfonts}
\usepackage{amssymb}
\usepackage[all]{xy}
\usepackage{graphicx}
\usepackage[nottoc, notlof, notlot]{tocbibind}
\usepackage{array}
\usepackage{url}
\usepackage{rotating}
\usepackage{fancyhdr}
\usepackage{fancyhdr}
\usepackage{titlesec}
\usepackage{xfrac} 
\titleformat{\section}[hang]{\center\Large\bf}{\thesection.}{0.5cm}{}

\DeclareSymbolFont{cyrletters}{OT2}{wncyr}{m}{n}
\DeclareMathSymbol{\Sha}{\mathalpha}{cyrletters}{"58}
\DeclareMathSymbol{\Brusse}{\mathalpha}{cyrletters}{"42}
\addto\captionsfrenchb{}

\theoremstyle{plain}
\newtheorem{theorem}{Th\'eor\`eme}[section]
\newtheorem{lemma}[theorem]{Lemme}
\newtheorem{proposition}[theorem]{Proposition}
\newtheorem{corollary}[theorem]{Corollaire}
\newtheorem{definition}[theorem]{D\'efinition}

\theoremstyle{definition}
\newtheorem{remarque}[theorem]{Remarque}

\newtheorem{example}[theorem]{Exemple}

\newtheoremstyle{hypo}  
  {\topsep}   
  {\topsep}   
  {\itshape}  
  {1.5ex}       
  {\bfseries} 
  {)}         
  {8pt plus 1pt minus 1pt}  
  {}          
\theoremstyle{hypo}

\newtheoremstyle{hypol}  
  {\topsep}   
  {\topsep}   
  {\itshape}  
  {1.5ex}       
  {\bfseries} 
  {)$_{\ell}$}         
  {8pt plus 1pt minus 1pt}  
  {}          
\theoremstyle{hypol}

\newtheoremstyle{DP}  
  {\topsep}   
  {\topsep}   
  {\itshape}  
  {}       
  {\bfseries} 
  {.}         
  {8pt plus 1pt minus 1pt}  
  {}          
\theoremstyle{DP}

\begin{document}

\title{\textbf{\scshape Dualité et principe local-global sur des corps locaux de dimension 2}}

\author{Diego Izquierdo\\
\small
École Normale Supérieure\\
\small
45, Rue d'Ulm - 75005 Paris - France\\
\small
\texttt{diego.izquierdo@ens.fr}
}
\date{}
\normalsize
\maketitle

\section*{Introduction}

\hspace{4ex} Annoncés dans les années 1960 par Tate, les théorèmes de dualité arithmétique ont joué depuis un rôle fondamental dans l'étude des points rationnels des variétés algébriques sur les corps globaux (corps de nombres et corps de fonctions de courbes sur des corps finis). Plus récemment, nous avons connu un regain d'intérêt pour ces questions sur des corps plus compliqués. C'est ainsi que Harari, Scheiderer et Szamuely ont étudié les corps de fonctions de courbes sur des corps $p$-adiques (\cite{HS1}, \cite{HS2}) , puis que Colliot-Thélène et Harari se sont penchés sur les corps de fonctions de courbes sur $\mathbb{C}((t))$ (\cite{CTH}). L'auteur du présent article a ensuite généralisé les précédents résultats aux corps de fonctions de courbes sur des corps locaux supérieurs (\cite{Izq1}, \cite{Izq2}, \cite{Izq3}). \\

\hspace{4ex} Soit $k$ un corps algébriquement clos de caractéristique 0, un corps $p$-adique, ou encore $\mathbb{C}((t))$. Considérons $K_0=k((x,y))$ le corps des séries de Laurent à deux variables sur $k$. On rappelle que $K_0$ est le corps des fractions de $k[[x,y]]$, un anneau local de dimension 2. Il est strictement contenu dans $k((x))((y))$ et, contrairement à $k((x))((y))$, son comportement est similaire à celui d'un corps global. Le but de cet article est d'obtenir des théorèmes de dualité arithmétique sur $K_0$, puis de les appliquer à l'étude des points rationnels sur les $K_0$-variétés algébriques. La principale difficulté que nous devons gérer dans cet article consiste à établir une dualité d'Artin-Verdier sur un tel corps puisqu'elle demande de travailler avec des schémas singuliers (typiquement la clôture intégrale de $\mathbb{C}[[x,y]]$ dans une extension finie de $\mathbb{C}((x,y))$). C'est un phénomène que l'on ne rencontrait pas dans les travaux portant sur les corps de fonctions de courbes, puisque dans ces cas-là la dualité d'Artin-Verdier découlait assez formellement de la dualité de Poincaré et d'un théorème de dualité pour le corps de base.\\

\hspace{4ex} Les motivations pour se pencher sur une telle question sont multiples. Tout d'abord, les corps de la forme $k((x,y))$ ont récemment intéressé plusieurs auteurs (Colliot-Thélène, Gille, Hu, Ojanguren, Parimala, Suresh...). Citons notamment l'article \cite{CTPS}, dans lequel Colliot-Thélène, Parimala et Suresh introduisent certaines obstructions au principe local-global sur $\mathbb{C}((x,y))$ et se demandent si ce sont les seules pour les espaces principaux homogènes sous des tores. Nous répondrons affirmativement à cette question dans le présent article. Ensuite, les corps de la forme $\mathbb{C}((x,y))$ interviennent de manière importante dans les travaux de Harbater, Hartmann et Kraschen dans lesquels ils utilisent des techniques de patching pour étudier le principe local-global sur certains corps de fonctions de courbes (par exemple \cite{HHK}). Il est donc naturel de penser que, si l'on veut comprendre les résultats qu'ils obtiennent avec les techniques de dualité, il convient de commencer par se pencher sur les questions de dualité pour $\mathbb{C}((x,y))$. Finalement, il est probablement nécessaire de comprendre le corps $\mathbb{C}((x,y))$ avant de s'intéresser aux corps de fonctions de surfaces complexes, pour lesquels rien n'est connu actuellement du point de vue de la dualité arithmétique.

\subsection*{Organisation de l'article.} 

\hspace{4ex} Ce texte est constitué de 3 parties. Dans la première section, on établit une dualité de type Artin-Verdier sur certains anneaux locaux de dimension 2. Plus précisément, on montre le théorème suivant:

\begin{theorem} (Corollaire \ref{AVdim2bis})\\
Soient $k$ un corps algébriquement clos de caractéristique 0 et $R_0$ une $k$-algèbre commutative, locale, intègre, normale, hensélienne, excellente, de dimension 2, de corps résiduel $k$. Notons $\mathfrak{m}_0$ l'idéal maximal de $R_0$ et $X_0 = \text{Spec}\; R_0 \setminus \{\mathfrak{m}_0\}$. Soient $j: U \hookrightarrow X_0$ une immersion ouverte avec $U$ non vide et $F$ un schéma en groupes fini étale sur $U$ de $n$-torsion et de dual de Cartier $F' = \underline{\text{Hom}}_U(F,\mu_n)$. Il existe alors un accouplement parfait de groupes finis:
$$H^r(U,F') \times H^{3-r}(X,j_!F) \rightarrow \mathbb{Z}/n\mathbb{Z}(-1).$$
\end{theorem}

Ici, $\mathbb{Z}/n\mathbb{Z}(-1)$ désigne $\text{Hom}(\mu_n,\mathbb{Z}/n\mathbb{Z})$, qui, quitte à choisir une racine primitive $n$-ième de l'unité, est isomorphe à $\mathbb{Z}/n\mathbb{Z}$. Le théorème, dont la preuve fait l'objet de la totalité de la section \ref{secAVdim2}, s'applique bien sûr à l'anneau $R_0=\mathbb{C}[[x,y]]$. La principale difficulté consiste à comprendre la cohomologie de $\mathbb{G}_m$ sur la normalisation de $R_0$ dans une extension finie de son corps des fractions, puisque cela demande de travailler avec des schémas singuliers, même lorsque $R_0$ est régulier. Ces problèmes sont résolus dans la proposition \ref{G_m}.\\

\hspace{4ex} Dans la deuxième section, nous appliquons le résultat précédent pour obtenir des théorèmes de dualité de type Poitou-Tate pour les modules finis et les tores sur des corps comme $\mathbb{C}((x,y))$ ou $\mathbb{C}((t))((x,y))$ ou $\mathbb{F}_q((x,y))$. Les principaux résultats sont les théorèmes \ref{PTC((x,t))} et \ref{PTC((x,t))tore}. \\

\hspace{4ex}  Finalement, dans la troisième section, nous appliquons les théorèmes de dualité précédents pour étudier les obstructions au principe local-global et à l'approximation faible. Plus précisément, dans les théorèmes \ref{PLGC((x,y))} et \ref{local-global}, nous comprenons les obstructions au principe local-global pour les torseurs sous des groupes linéaires connexes sur $\mathbb{C}((x,y))$ ainsi que pour les torseurs sous des tores sur $\mathbb{F}_q((x,y))$, $\mathbb{C}((t))((x,y))$ ou $\mathbb{Q}_p((x,y))$. Ce sont précisément ces deux théorèmes (\ref{PLGC((x,y))} et \ref{local-global}) qui répondent à la question posée à la fin de l'article \cite{CTPS}. Ils impliquent aussi la proposition suivante:
\begin{proposition} (Proposition \ref{stabrat})\\
Soit $k$ un corps fini de caractéristique $p$ et considérons un tore stablement rationnel $T$ sur $K_0=k((x,y))$. Soit $Y$ un espace principal homogène sous $T$ qui devient trivial sur une extension finie de $K_0$ de degré non divisible par $p$. Alors $Y$ vérifie le principe local-global.
\end{proposition}

\begin{remarque}
Une proposition analogue à la précédente est vraie sur un corps de la forme $\mathbb{Q}_p(x)$ (corollaire 5.7 de \cite{HS1}), mais pas sur un corps de la forme $\mathbb{C}((t))(x)$ (exemple 2.9 de \cite{CTH}).
\end{remarque}

À la toute fin de l'article, dans le théorème \ref{AFC((x,y))}, nous nous penchons sur l'approximation faible pour les tores sur $\mathbb{C}((x,y))$, $\mathbb{C}((t))((x,y))$ ou encore sur $\mathbb{Q}_p((x,y))$.  \\

\hspace{4ex} Le lecteur remarquera que certaines preuves en fin d'article sont similaires à des preuves pouvant être trouvées dans la littérature, auquel cas elles ne sont qu'esquissées.\\

\hspace{4ex} \textbf{Remerciements.} Je tiens à remercier en premier lieu David Harari pour son soutien et ses conseils, ainsi que pour sa lecture soigneuse de ce texte: sans lui, ce travail n'aurait pas pu voir le jour. Je suis aussi très reconnaissant à Olivier Wittenberg pour d'intéressantes discussions. Je voudrais finalement remercier l'École Normale Supérieure pour ses excellentes conditions de travail.

\subsection*{Rappels et notations}

\textbf{Corps.} Si $L$ est un corps, on notera $L^s$ sa clôture séparable. \\

\textbf{Dimension.} Lorsque $Z$ est un schéma noethérien et $i$ un entier naturel, on notera $Z^{(i)}$ l'ensemble des points de codimension $i$ de $Z$. En particulier, $Z^{(0)}$ désignera l'ensemble des points génériques de $Z$, que l'on confondra parfois avec l'ensemble des composantes irréductibles de $Z$.\\

\textbf{Groupes abéliens.} Pour $M$ un groupe abélien, $n>0$ un entier et $\ell$ un nombre premier, on notera:
\begin{itemize}
\item[$\bullet$] ${_n}M$ la partie de $n$-torsion de $M$.
\item[$\bullet$] $M\{\ell\}$ la partie de torsion $\ell$-primaire de $M$.
\item[$\bullet$] $M_{\text{non}-\ell}=\bigoplus_{p \neq \ell} M\{p\}$ où $p$ décrit les nombres premiers différents de $\ell$. Par commodité, $M_{\text{non}-0}$ désignera $ M$.
\item[$\bullet$] $M^{(\ell)} = \varprojlim_n M/\ell^n$ le complété $\ell$-adique de $M$.
\item[$\bullet$] $\overline{M}$ le quotient de $M$ par son sous-groupe divisible maximal.
\item[$\bullet$] $M^D$ le groupe des morphismes $M \rightarrow \mathbb{Q}/\mathbb{Z}$.
\end{itemize}
\vspace{5 mm}

\textbf{Faisceaux.} Sauf indication du contraire, tous les faisceaux sont considérés pour le petit site étale. Pour $F$ un faisceau sur un schéma $Z$, on note $\text{Hom}_Z(F,-)$ (ou $\text{Hom}(F,-)$ s'il n'y a pas d'ambiguïté) le foncteur qui à un faisceau $G$ sur $Z$ assosie le groupe des morphismes de faisceaux de $F$ vers $G$ et $\underline{\text{Hom}}_Z(F,-)$ (ou $\underline{\text{Hom}}(F,-)$ s'il n'y a pas d'ambiguïté) le foncteur qui à un faisceau $G$ sur $Z$ associe le faisceau étale $U \mapsto \text{Hom}_U(F|_U,G|_U)$. \\

\textbf{Cohomologie.} Par convention, pour $F$ et $G$ des faisceaux sur un schéma $X$ et $r$ un entier strictement négatif, on pose $\text{Ext}^r_X(F,G)=H^r(X,F)=0$.\\

\textbf{Groupe de Brauer.} Lorsque $Z$ est un schéma, on note $\text{Br}(Z)$ le groupe de Brauer cohomologique $H^2(Z,\mathbb{G}_m)$. Si $Z$ est une $L$-variété géométriquement intègre pour un certain corps $L$, on note $\text{Br}_{\text{al}}(Z)$ le quotient du groupe de Brauer algébrique $\text{Ker}(\text{Br}(Z) \rightarrow \text{Br}(Z \times_L L^s))$ par $\text{Im}(\text{Br}(L) \rightarrow \text{Br}(Z))$.\\

\textbf{Complexes de Bloch et cohomologie motivique.} Lorsque $X$ est un schéma de Dedekind, on notera $\mathbb{Z}(i)$ le $i$-ème complexe de Bloch sur le petit site étale de $X$. On renvoie à la partie 0.5 de \cite{Izq1} pour des rappels sur les complexes de Bloch.\\

\textbf{Courbes sur un corps algébriquement clos.} Soient $k$ un corps algébriquement clos et $Y$ une courbe projective réduite connexe sur $k$. La normalisation de $Y$ est, par définition, la réunion disjointe des normalisations des composantes irréductibles de $Y$. Si $\tilde{Y}$ est la normalisation de $Y$, le lemme 7.5.18(b) de \cite{Liu} donne une suite exacte:
$$0 \rightarrow (k^{\times})^{t} \rightarrow \text{Pic} \; Y \rightarrow \text{Pic} \; \tilde{Y} \rightarrow 0,$$
pour un certain entier naturel $t$. De plus, toujours d'après le lemme 7.5.18(b) de \cite{Liu}, on a $t = \mu - c + 1$ où $c$ est le nombre de composantes irréductibles de $Y$ et $\mu =\sum_{y \in Y(k) } (m_y - 1)$, l'entier $m_y$ désignant le nombre de points de $\tilde{Y}$ qui sont au-dessus de $y$. Par ailleurs, d'après le théorème 7.4.39 de \cite{Liu}, le groupe $\text{Pic} \; \tilde{Y}$ est isomorphe à $ \mathbb{Z}^{|Y^{(0)}|} \oplus (\mathbb{R}/\mathbb{Z})^{2\cdot \sum_{v \in Y^{(0)}} g_v}$, où $g_v$ désigne le genre de la composante irréductible de $Y$ de point générique $v$. Comme $k^{\times}$ est divisible, la suite exacte précédente est scindée, d'où:
\begin{equation*}\label{Picbis}
\text{Pic}\; Y \cong (k^{\times})^{t} \oplus \mathbb{Z}^{|Y^{(0)}|} \oplus (\mathbb{R}/\mathbb{Z})^{2\cdot \sum_{v \in Y^{(0)}} g_v}.
\end{equation*}

\textbf{Graphes.} On utilisera notamment les notions de graphe biparti et nombre cyclomatique d'un graphe. Par graphe biparti on entend graphe de nombre chromatique 2 au sens du deuxième paragraphe du chapitre IV de \cite{Berge}. Le nombre cyclomatique est défini dans le premier paragraphe du chapitre IV de \cite{Berge}.

\subsection*{Les corps étudiés}

\hspace{4ex} Soient $k$ un corps et $R_0$ une $k$-algèbre commutative, locale, géométriquement intègre (ie $R_0 \otimes_k k^s$ est intègre), normale, hensélienne, excellente, de dimension 2 et de corps résiduel $k$. Par exemple, $R_0$ pourrait être le complété ou l'hensélisé d'une surface complexe normale en un de ses points fermés. Soient $K_0$ le corps des fractions de $R_0$ et $\mathfrak{m}_0$ l'idéal maximal de $R_0$. Posons $\mathcal{X}_0 = \text{Spec} \; R_0$ et $X_0 = \mathcal{X}_0 \setminus \{\mathfrak{m}_0\}$. Le schéma $\mathcal{X}_0$ est, en général, singulier, alors que $X_0$ est un schéma de Dedekind (non affine).\\

\hspace{4ex} Supposons donnés $\tilde{\mathcal{X}}_0$ un schéma régulier intègre de dimension 2 et un morphisme projectif surjectif $f_0: \tilde{\mathcal{X}}_0 \rightarrow \mathcal{X}_0$ induisant un isomorphisme entre $f_0^{-1}(X_0)$ et $X_0$. Si $f_0$ n'est pas un isomorphisme, la fibre spéciale $f^{-1}_0(\mathfrak{m}_0)$ de $f_0$ est une $k$-courbe en général réductible. Chaque $v \in (\tilde{\mathcal{X}}_0)^{(1)}$ définit une valuation discrète de rang 1 sur $K_0$. On note alors $K_{0,v}$ le complété correspondant de $K_0$ et $k(v)$ son corps résiduel. Si $v \not\in f^{-1}_0(\mathfrak{m}_0)$, le corps $k(v)$ est le corps des fractions d'un anneau de valuation discrète hensélien de corps résiduel $k$. Si $v \in f_0^{-1}(\mathfrak{m}_0)$, le corps $k(v)$ est le corps des fonctions de la composante irréductible de $f^{-1}_0(\mathfrak{m}_0)$ correspondant à $v$.\\

\hspace{4ex} Dans le cas où $k$ est algébriquement clos de caractéristique nulle, le corps $K_0$ est de dimension cohomologique 2, et les $k(v)$ sont de dimension cohomologique 1. De plus, si $v \not\in f^{-1}_0(\mathfrak{m}_0)$, le groupe de Galois absolu de $k(v)$ est isomorphe à $\hat{\mathbb{Z}}$ et, d'après l'exemple I.1.10 de \cite{MilADT}, pour chaque module galoisien fini $F$ sur $k(v)$, on a une dualité parfaite de groupes finis:
$$H^0(k(v),F) \times H^1(k(v),\underline{\text{Hom}}(F,k(v)^{s,\times})) \rightarrow \mathbb{Q}/\mathbb{Z}.$$
Dans ce contexte, la remarque 2.3 de \cite{CTPS} fournit une suite exacte:
$$0 \rightarrow \text{Br} \; K_0 \rightarrow \bigoplus_{v \in (\tilde{\mathcal{X}}_0)^{(1)}} \text{Br} \; K_{0,v} \rightarrow \bigoplus_{v \in (\tilde{\mathcal{X}}_0)^{(2)}} \mathbb{Q}/\mathbb{Z} \rightarrow 0,$$
puisque l'application résidu $H^2(K_{0,v},\mu_n) \rightarrow H^1(k(v),\mathbb{Z}/n\mathbb{Z})$ (définie dans l'appendice du chapitre II de \cite{Ser}) est un isomorphisme.\\

\hspace{4ex} Pour terminer, on rappelle aussi que, dans le cas où $k$ est séparablement clos, l'indice et l'exposant pour les $K_0$-algèbres simples centrales coïncident et que la conjecture de Serre II vaut pour $K_0$. On pourra aller voir le théorème 1.4 de \cite{CGR}.

\section{Dualité d'Artin-Verdier}\label{secAVdim2}

\hspace{4ex} Dans toute cette section, on supposera que $k$ est algébriquement clos de caractéristique 0. Le but est de construire pour chaque entier $r\in \mathbb{Z}$ et chaque faisceau constructible $F$ sur un ouvert non vide $U$ de $X_0$ un accouplement:
$$ AV: \text{Ext}^r_U(F,\mathbb{G}_m) \times H^{3-r}(X_0,j_!F) \rightarrow \mathbb{Q}/\mathbb{Z}$$
(où $j: U \hookrightarrow X_0$ désigne l'immersion ouverte) et de démontrer qu'il s'agit d'un accouplement parfait de groupes finis. Nous allons suivre le même schéma de preuve que pour le théorème classique d'Artin-Verdier sur les corps de nombres (voir la section II.3 de \cite{MilADT}). Cependant, pour ce faire, il est nécessaire de commencer par comprendre la cohomologie du faisceau $\mathbb{G}_m$ sur les ouverts de $X_0$. Et cela est nettement plus compliqué que dans le cas des corps de nombres parce qu'il faut gérer les singularités de $\mathcal{X}_0$. C'est une des principales nouveautés apportées par cet article.

\begin{remarque}
\begin{itemize}
\item[$\bullet$] Dans le cas $K_0=\mathbb{C}((x,y))$, on pourrait penser que l'on n'a pas besoin de gérer des singularités, puisque $\mathbb{C}[[x,y]]$ est régulier. Cependant, dans la preuve du théorème d'Artin-Verdier, on a besoin de pouvoir remplacer $K_0$ par des extensions finies, et la clôture intégrale de $\mathbb{C}[[x,y]]$ dans une extension finie de $\mathbb{C}((x,y))$ est, en général, singulière.
\item[$\bullet$] Dans les articles \cite{HS1}, \cite{CTH} et \cite{Izq1}, la dualité d'Artin-Verdier ne posait pas trop de difficultés puisqu'elle découlait assez formellement de la dualité de Poincaré et d'une dualité sur le corps de base. C'est une différence majeure entre cet article et les articles antérieurs portant sur les théorèmes de dualité sur des corps de fonctions de courbes.
\item[$\bullet$] Il y a très peu de résultats dans la littérature de dualités à la Artin-Verdier dans des cas où il faille considérer des singularités. On peut citer la section II.6 de \cite{MilADT}, mais la situation considérée et les preuves sont très différentes de celles du présent article.
\end{itemize}
\end{remarque}

\subsection{Préliminaire: deux lemmes d'algèbre linéaire}

\hspace{4ex} Dans ce paragraphe, on établit deux lemmes d'algèbre linéaire sur les graphes, qui seront très utiles par la suite.

\begin{lemma} \label{graphe}
Soit $A$ un groupe abélien. Soit $\Gamma$ un graphe fini connexe non orienté biparti. On note $V=V_1 \sqcup V_2$ l'ensemble de ses sommets et $E$ l'ensemble de ses arêtes. Pour $v \in V$, on note $I(v)$ l'ensemble des arêtes ayant $v$ pour extrémité. Soit $(a_v)_{v \in V} \in A^V$ tel que:
$$\sum_{v \in V_1} a_v =\sum_{v \in V_2} a_v.$$
Alors il existe $(x_e)_{e \in E}\in A^E$ tel que, pour tout $v \in V$, on a: $$\sum_{e \in I(v)} x_e = a_v.$$
\end{lemma}

\begin{proof}
Considérons un sous-ensemble d'arêtes $E' \subseteq E$ tel que le sous-graphe $\Gamma'$ de $\Gamma$ dont l'ensemble des sommets est $V$ et celui des arêtes est $E'$ soit un arbre connexe. En posant $x_e=0$ pour $e \in E'$, on se ramène à prouver le lemme pour $\Gamma'$ au lieu de $\Gamma$. On peut donc supposer que $\Gamma$ est un arbre.\\ On procède alors par récurrence sur $|V|$. Si $|V|=1$, le résultat est évident. Soit maintenant $n \geq 1$ et supposons le résultat prouvé dans le cas $|V| = n$. Supposons que $|V|=n+1$. Soit $\Gamma'$ le sous-graphe de $\Gamma$ obtenu en enlevant une feuille $v_0 \in V$ et l'unique arête $e_0 \in E$ qui lui est incidente. Soit $v_1$ l'autre extrémité de $e_0$ et considérons la famille $(b_v)_{v \in V \setminus \{v_0\}}$ définie par $b_v = a_v$ pour $v \not\in \{v_0,v_1\}$ et $b_{v_1}=a_{v_1} - a_{v_0}$. Par hypothèse de récurrence, il existe une famille $(x_e)_{e \in E \setminus \{e_0\}}$ telle que, pour chaque $v \in  V \setminus \{v_0\}$, on a:
$$\sum_{e \in I(v) \setminus \{e_0\}} x_e = b_v.$$
On pose $x_{e_0} = a_{v_0}$. On a alors, pour tout $v \in V$: $$\sum_{e \in I(v)} x_e = a_v.$$
\end{proof}

\begin{lemma} \label{graphe2}
On reprend les notations du lemme précédent. Soit $c$ le nombre cyclomatique (ou premier nombre de Betti) de $\Gamma$. Pour $v \in V$, on note $N(v)$ l'ensemble des voisins de $v$. Lorsque $S$ est un ensemble, on note $\Sigma_S$ le morphisme somme $\bigoplus_{s \in S} A \rightarrow A$.  On pose:
$$\Theta_A(\Gamma):=\left( \bigoplus_{v \in V_1} \text{Ker} \; \Sigma_{N(v)} \right) \cap \left( \bigoplus_{v \in V_2} \text{Ker} \; \Sigma_{N(v)} \right)  \subseteq \bigoplus_{E} A.$$
\begin{itemize}
\item[(i)] Soit $e_0 \in E$ une arête de $\Gamma$ telle que le sous-graphe de $\Gamma$ obtenu en enlevant $e_0$ est connexe. Notons $p_{e_0}^\Gamma: \bigoplus_{E} A \rightarrow A$ la projection sur la composante indexée par $e_0$. Alors la restriction de $p_{e_0}^{\Gamma}$ à $\Theta_A(\Gamma)$ est surjective et scindée.
\item[(ii)] Le groupe $\Theta_A(\Gamma)$ est isomorphe à $A^{c}$.
\end{itemize}
\end{lemma}

\begin{proof}
\begin{itemize}
\item[(i)] Procédons par récurrence sur $c$.
\begin{itemize}
\item[$\bullet$] Supposons que $c = 1$. Soit $\Gamma'$ l'unique sous-graphe cyclique de $\Gamma$. On remarque alors immédiatement que $\Theta_A(\Gamma) = \Theta_A(\Gamma')\cong A$ et que $p_{e_0}^{\Gamma}|_{\Theta_A(\Gamma)}$ s'identifie à l'identité sur $A$.
\item[$\bullet$] Soit $c_0 \geq 1$ et supposons le résultat prouvé pour $c = c_0$. Supposons que $\Gamma$ est de nombre cyclomatique $c_0+1$. Comme $\Gamma$ est de nombre cyclomatique au moins 2, il existe une arête $e_1$ telle que le sous-graphe de $\Gamma$ obtenu en enlevant les arêtes $e_0$ et $e_1$ est connexe. Soit alors $\Gamma'$ le sous-graphe de $\Gamma$ obtenu en enlevant l'arête $e_1$. Par hypothèse de récurrence, $p_{e_0}^{\Gamma'}|_{\Theta_A(G')}$ est un morphisme surjectif et scindé. On en déduit immédiatement que $p_{e_0}^{\Gamma}|_{\Theta_A(\Gamma)}$ l'est aussi.
\end{itemize}

\item[(ii)] Procédons par récurrence sur $c$. 
\begin{itemize}
\item[$\bullet$] Supposons que $c=0$, c'est-à-dire que $\Gamma$ est un arbre. Soit $(x_e)_{e \in E} \in \Theta_A(\Gamma)$. On remarque alors que, si $f$ une feuille de $\Gamma$ et si $e_f$ est l'unique arête de $\Gamma$ ayant $f$ pour extrémité, alors $x_{e_f}=0$. Une récurrence simple sur le nombre de sommets de $\Gamma$ permet donc de conclure que $x_e=0$ pour tout $e \in E$, autrement dit, que $\Theta_A(\Gamma)=0$.
\item[$\bullet$] Soit $c_0 \geq 0$ et supposons le résultat prouvé pour $c = c_0$. Supposons que $\Gamma$ est de nombre cyclomatique $c_0+1$. Considérons alors un sous-graphe $\Gamma'$ de $\Gamma$, connexe, de nombre cyclomatique $c_0$, obtenu à partir de $\Gamma$ en enlevant une arête $e_0$. On remarque alors immédiatement que l'on a une suite exacte:
$$0 \rightarrow \Theta_A(\Gamma') \rightarrow \Theta_A(\Gamma) \rightarrow A$$
où le morphisme $\Theta_A(\Gamma) \rightarrow A$ est induit par la projection $p_{e_0}^{\Gamma}$. Par conséquent, en utilisant (i) et l'hypothèse de récurrence, on obtient des isomorphismes: $$\Theta_A(\Gamma) \cong \Theta_A(\Gamma') \oplus A \cong A^{c_0+1}.$$
\end{itemize}
\end{itemize}
\end{proof}

\subsection{Résolution des singularités}\label{desingu}

\hspace{4ex} Soient $K$ une extension finie de $K_0$ et $\mathcal{O}_K$ la normalisation de $R_0$ dans $K$. Comme l'anneau $R_0$ est local hensélien excellent de dimension 2, il en est de même de $\mathcal{O}_K$ (scholie 7.8.3(ii) de \cite{EGA42} et proposition 18.5.9(i) de \cite{EGA44}). On note alors $\mathfrak{m}$ l'idéal maximal de $\mathcal{O}_K$ et on pose $\mathcal{X} = \text{Spec}(\mathcal{O}_K)$ et $X = \mathcal{X} \setminus \{\mathfrak{m}\}$. On notera aussi $\eta = \text{Spec} \; K$ et $g: \eta \hookrightarrow X$ le point générique de $X$.

\hspace{4ex} Considérons un morphisme de schémas $f: \tilde{\mathcal{X}} \rightarrow \mathcal{X} = \text{Spec} \; \mathcal{O}_K$ vérifiant les hypothèses suivantes:
\begin{itemize}
\item[$\bullet$] $\tilde{\mathcal{X}}$ est intègre régulier de dimension 2 et $f$ est projectif;
\item[$\bullet$] $f: f^{-1}(X) \rightarrow X$ est un isomorphisme;
\item[$\bullet$] $f^{-1}(\mathfrak{m})$ est un diviseur à croisements normaux de $\mathcal{X}$ (au sens de la définition 9.1.6 de \cite{Liu}).
\end{itemize}

Un tel morphisme existe d'après \cite{Lipbis}. On pose alors $Y = f^{-1}(\mathfrak{m})$ muni de la structure réduite. Ainsi, $Y$ est une $k$-courbe réduite qui n'est pas forcément irréductible, mais dont les composantes irréductibles sont lisses. Pour $v \in \tilde{\mathcal{X}}^{(1)} \setminus X^{(1)} = Y^{(0)}$, on désigne par $Y_v$ la $k$-courbe projective lisse correspondant à $v$. On note $g_v$ le genre de cette dernière.\\

\hspace{4ex}Par ailleurs, à la courbe $Y$ on associe le graphe biparti $\Gamma$ suivant:
\begin{itemize}
\item[$\bullet$] sommets: $V= V_1 \sqcup V_2$, où $V_1 = Y^{(0)}$ est l'ensemble des (points génériques des) composantes irréductibles de $Y$ et $V_2$ est l'ensemble des points fermés de $Y$ qui sont intersection de deux composantes irréductibles de $Y$;
\item[$\bullet$] arêtes: $E$ est l'ensemble des couples $(v_1,v_2) \in V_1 \times V_2$ tels que $v_2\in Y_{v_1}$.
\end{itemize}
Soit $c_{\Gamma}$ le nombre cyclomatique du graphe $\Gamma$. On pose finalement:
$$n_X= 2 \sum_{v \in \tilde{\mathcal{X}}^{(1)} \setminus X^{(1)}} g_v + c_{\Gamma}.$$
A priori, la quantité $n_X$ dépend de $f$, mais nous verrons dans la suite qu'en fait elle ne dépend que de $X$.

\begin{remarque}\label{deuxgraphes}
Usuellement ce n'est pas le graphe $\Gamma$ qu'on associe à $Y$ mais le graphe $\Gamma_{\text{réd}}$ défini par:
\begin{itemize}
\item[$\bullet$] sommets: $V_{\text{réd}} = Y^{(0)}$;
\item[$\bullet$] arêtes: $E_{\text{réd}}$ est l'ensemble des couples $(v_1,v_2) \in V_{\text{réd}}^2$ tels que $Y_{v_1}$ et $Y_{v_2}$ s'intersectent.
\end{itemize}
Le graphe $\Gamma$ est obtenu à partir de $\Gamma_{\text{réd}}$ est rajoutant un sommet sur chaque arête. En particulier, les deux graphes ont même nombre cyclomatique.
\end{remarque}

\subsection{Les modules $\Lambda$ et $\Upsilon$}

\hspace{4ex} On rappelle que, pour $v \in \tilde{\mathcal{X}}^{(1)}$ et $w \in \tilde{\mathcal{X}}^{(2)}$ , le corps $k(Y_v)_w$ est isomorphe à $k((t))$. Son groupe de Galois absolu est isomorphe à $\hat{\mathbb{Z}}$, et donc on a un isomorphisme naturel $H^1(k(Y_v),\mathbb{Q}/\mathbb{Z}(1)) \cong \mathbb{Q}/\mathbb{Z}$.
On peut alors considérer le morphisme naturel:
$$\phi: \bigoplus_{v \in \tilde{\mathcal{X}}^{(1)} \setminus X^{(1)}} H^1(k(Y_v),\mathbb{Q}/\mathbb{Z}(1)) \rightarrow \bigoplus_{w \in \tilde{\mathcal{X}}^{(2)}} \mathbb{Q}/\mathbb{Z},$$
induit par les morphismes de restriction:
$$H^1(k(Y_v),\mathbb{Q}/\mathbb{Z}(1)) \rightarrow H^1(k(Y_v)_w,\mathbb{Q}/\mathbb{Z}(1)) \cong \mathbb{Q}/\mathbb{Z},$$
pour $v \in Y^{(0)}$ et $w \in Y_v^{(1)}$. On pose $\Lambda = \text{Coker} \; \phi$ et $\Upsilon = \text{Ker} \; \phi$.

\begin{lemma}\label{Lambda}
Le morphisme somme $\Sigma: \bigoplus_{w \in \tilde{\mathcal{X}}^{(2)}} \mathbb{Q}/\mathbb{Z} \rightarrow \mathbb{Q}/\mathbb{Z}$ induit un isomorphisme $\Lambda \cong \mathbb{Q}/\mathbb{Z}$.
\end{lemma}

\begin{proof}
Le morphisme $\phi$ est induit par les morphismes:
$$\phi_v: H^1(k(Y_v),\mathbb{Q}/\mathbb{Z}(1)) \rightarrow \bigoplus_{w \in Y_v^{(1)}} H^1(k(Y_v)_w,\mathbb{Q}/\mathbb{Z}(1)).$$
Or, d'après le théorème 2.9 de \cite{Izq1} dans le cas $d=-1$, on a une suite exacte:\\
\begin{equation}\label{cbecomplexe}
\xymatrix{
H^1(k(Y_v),\mathbb{Q}/\mathbb{Z}(1)) \ar[r]^-{\phi_v} & \bigoplus_{w \in Y_v^{(1)}} H^1(k(Y_v)_w,\mathbb{Q}/\mathbb{Z}(1)) \ar[r]^-{\Sigma} & \mathbb{Q}/\mathbb{Z} \ar[r] & 0.
}\tag{$\star$}
\end{equation}
Cela montre immédiatement que $\Sigma$ induit un morphisme surjectif:
$$\Lambda \rightarrow \mathbb{Q}/\mathbb{Z}.$$
Reste à prouver qu'il est injectif. Pour ce faire, on se donne $\alpha=(\alpha_w)_{w \in \tilde{\mathcal{X}}^{(2)}} \in \text{Ker}\; \Sigma$ et on considère le graphe biparti $\Gamma$. Pour $v_1 \in V_1$, on pose $a_{v_1} = -\sum_{w \in Y_{v_1}^{(1)}\setminus V_2} \alpha_w$. Pour $v_2 \in V_2$, on pose $a_{v_2} = \alpha_{v_2}$. Comme $(\alpha_w) \in \text{Ker} \; \Sigma$, on vérifie immédiatement que: $$\sum_{v \in V_1} a_v =\sum_{v \in V_2} a_v.$$
De plus, le graphe $\Gamma$ est connexe d'après le principe de connexité de Zariski (corollaire 5.3.16 de \cite{Liu}). Par conséquent, le lemme \ref{graphe} montre qu'il existe $(x_e)_{e \in E} \in (\mathbb{Q}/\mathbb{Z})^E$ tel que, pour tout $v \in V$, on a: $$\sum_{e \in I(v)} x_e = a_v.$$
Pour $v_1 \in V_1$, on considère la famille $y_{v_1}=(y_{v_1,w})_{w \in \tilde{\mathcal{X}}^{(2)}}$ définie par:
\begin{itemize}
\item[$\bullet$] $y_{v_1,w} = 0$ si $w \not\in Y_{v_1}$;
\item[$\bullet$] $y_{v_1,w} = \alpha_w$ si $w \in Y_{v_1} \setminus V_2$;
\item[$\bullet$] $y_{v_1,w} = x_{(v_1,w)}$ si $(v_1,w) \in E$.
\end{itemize}
On remarque alors que, pour chaque $v_1 \in V_1$, on a $\sum_{w \in \tilde{\mathcal{X}}^{(2)}} y_{v_1,w} = 0$. Par conséquent, la suite exacte (\ref{cbecomplexe}) montre que $y \in \text{Im}(\phi_{v_1})$. Comme $\sum_{v \in V_1} y_v = \alpha$, on en déduit que $\alpha \in \text{Im}(\phi)$ et $\Sigma$ induit bien un isomorphisme $\Lambda \rightarrow \mathbb{Q}/\mathbb{Z}$.
\end{proof}

\begin{lemma}\label{Upsilon}
Le groupe abélien $\Upsilon$ est isomorphe à $(\mathbb{Q}/\mathbb{Z})^{n_X}$.
\end{lemma}

\begin{proof}
En reprenant les notations du lemme précédent et en utilisant la suite (\ref{cbecomplexe}), on remarque que l'on dispose d'une suite exacte:
$$0 \rightarrow \bigoplus_{v \in \tilde{\mathcal{X}}^{(1)} \setminus X^{(1)}} \text{Ker} \; \phi_v \rightarrow \Upsilon \rightarrow \Theta \rightarrow 0,$$
où: $$\Theta = \text{Ker}\left( \Sigma: \bigoplus_{v \in \tilde{\mathcal{X}}^{(1)} \setminus X^{(1)}} \text{Ker} \left( \Sigma: \bigoplus_{w \in Y_v^{(1)}} \mathbb{Q}/\mathbb{Z} \rightarrow \mathbb{Q}/\mathbb{Z}\right) \rightarrow \bigoplus_{w \in \mathcal{X}^{(1)}} \mathbb{Q}/\mathbb{Z} \right).$$
Or:
\begin{itemize}
\item[$\bullet$] pour $v \in \tilde{\mathcal{X}}^{(1)} \setminus X^{(1)}$, le morphisme $\phi_v$ s'identifie au morphisme:
$$H^1(k(Y_v),\mathbb{Q}/\mathbb{Z}(1)) \rightarrow \bigoplus_{w \in Y_v^{(1)}} H^0(k,\mathbb{Q}/\mathbb{Z})$$
 induit par les résidus (dont on peut trouver la définition dans l'appendice du chapitre II de \cite{Ser}), et donc $\text{Ker} \; \phi_v = H^1(Y_v, \mathbb{Q}/\mathbb{Z}(1)) \cong (\mathbb{Q}/\mathbb{Z})^{2g_v}$;
 \item[$\bullet$] le groupe abélien $\Theta$ est en fait $\Theta_{\mathbb{Q}/\mathbb{Z}}(\Gamma)$ et donc, d'après le lemme \ref{graphe2}, il est isomorphe à $(\mathbb{Q}/\mathbb{Z})^{c_{\Gamma}}$.
\end{itemize}
On en déduit que $\Upsilon \cong (\mathbb{Q}/\mathbb{Z})^{n_X}$.
\end{proof}

\subsection{Accouplement d'Artin-Verdier}

\hspace{4ex} Fixons un ouvert $U$ de $X$ non vide, ainsi qu'un faisceau constructible $F$ sur $U$. Soit $j: U \rightarrow X$ l'immersion ouverte et posons $H^r_c(U,F) = H^r(X,j_!F)$ pour $r \geq 0$. Si $D(U)$ désigne la catégorie dérivée bornée des faisceaux sur le petit site étale de $U$, en identifiant $\text{Ext}^r_U(F,\mathbb{G}_m) = \text{Hom}_{D(U)}(F,\mathbb{G}_m[r])$, on obtient un accouplement naturel:
$$ AV: \text{Ext}^r_U(F,\mathbb{G}_m) \times H^{3-r}_c(U,F) \rightarrow H^3_c(U,\mathbb{G}_m).$$

\begin{remarque}
En prenant $K=K_0$, on obtient un accouplement:
$$ AV: \text{Ext}^r_U(F,\mathbb{G}_m) \times H^{3-r}(X_0,j_!F) \rightarrow H^3_c(U,\mathbb{G}_m)$$
pour chaque immersion ouverte $j: U \hookrightarrow X_0$ avec $U$ non vide et chaque faisceau constructible $F$ sur $U$.
\end{remarque}

\begin{proposition} \textbf{(Cohomologie de $\mathbb{G}_m$)} \label{G_m}
\begin{itemize}
\item[(i)] On a une suite exacte:
$$0 \rightarrow H^2(U,\mathbb{G}_m) \rightarrow \text{Br}(K) \rightarrow \bigoplus_{v\in U^{(1)}} \text{Br}(K_v) \rightarrow H^3(U,\mathbb{G}_m) \rightarrow 0.$$
De plus, $H^r(U,\mathbb{G}_m)=0$ pour $r>3$.
\item[(ii)] Le groupe $H^3_c(U,\mathbb{G}_m)$ est isomorphe à $\mathbb{Q}/\mathbb{Z}$. 
\item[(iii)] Le groupe $\text{Br} \; X$ est isomorphe à $(\mathbb{Q}/\mathbb{Z})^{n_X}$.
\item[(iv)] Pour tout entier naturel non nul $m$, les groupes ${_m}\text{Pic}\; X$ et $\text{Pic}\; X /m$ sont finis et on a:
$$\frac{|{_m}\text{Pic}\; X|}{|\text{Pic} \; X /m|} = m^{n_X}.$$
\end{itemize}
\end{proposition}

\begin{remarque}
En particulier, $n_X$ ne dépend que de $X$.
\end{remarque}

\begin{proof} 
\begin{itemize}
\item[(i)] On dispose d'une suite exacte de faisceaux sur $U$: $$0 \rightarrow \mathbb{G}_m \rightarrow g_*\mathbb{G}_m \rightarrow \text{Div}_U \rightarrow 0,$$
où $g:  \eta \hookrightarrow U$ désigne l'inclusion du point générique et $\text{Div}_U$ le faisceau des diviseurs sur $U$ (le lecteur pourra trouver sa définition dans l'exemple II.3.9 de \cite{MilADT}). La suite exacte longue associée s'écrit:
$$ ... \rightarrow H^r(U,\mathbb{G}_m) \rightarrow H^r(U,g_*\mathbb{G}_m) \rightarrow H^r(U,\text{Div}_U) \rightarrow ...$$
Si $i_v$ désigne l'immersion fermée $\text{Spec} \; (k(v)) \hookrightarrow U$, on a $\text{Div}_U = \bigoplus_{v \in U^{(1)}} {i_v}_*\mathbb{Z}$, et donc, comme $i_v$ est un morphisme fini: 
$$H^r(U,\text{Div}_U) = \bigoplus_{v \in U^{(1)}} H^r(U,{i_v}_*\mathbb{Z})  \cong \bigoplus_{v \in U^{(1)}} H^r(k(v),\mathbb{Z}).$$
Par ailleurs, pour $s>0$, la tige de $R^sg_*\mathbb{G}_m$ en un point géométrique $\overline{\eta}$ d'image $\eta$ est $H^s(K^s,\mathbb{G}_m) = 0$ et, pour $v \in U \setminus \{ \eta \}$, la tige de $R^sg_*\mathbb{G}_m$ en un point géométrique $\overline{v}$ d'image $v$ est $H^s(K_v^{nr},\mathbb{G}_m)$, qui vaut 0 pour $s=1$ d'après le théorème de Hilbert 90 et qui vaut 0 pour $s>1$ parce que $K_v^{nr}$ est un corps de dimension cohomologique au plus 1 (exemple II.3.3.c de \cite{Ser}). On en déduit que $R^sg_*\mathbb{G}_m = 0$. La suite spectrale de Leray:
$$H^r(U,R^sg_*\mathbb{G}_m) \Rightarrow H^{r+s}(K,\mathbb{G}_m)$$
fournit alors des isomorphismes $H^r(U,g_*\mathbb{G}_m) \cong H^r(K,\mathbb{G}_m)$.
Par conséquent, on obtient la suite exacte:
$$0 \rightarrow H^2(U,\mathbb{G}_m) \rightarrow \text{Br}(K) \rightarrow \bigoplus_{v\in U^{(1)}} \text{Br}(K_v) \rightarrow H^3(U,\mathbb{G}_m) \rightarrow 0$$
ainsi que la trivialité de $H^r(U,\mathbb{G}_m)$ pour $r>3$ puisque $K$ est de dimension cohomologique 2.
\item[(ii)] Le schéma $\tilde{\mathcal{X}}$ étant régulier, d'après la remarque 2.3 de \cite{CTPS}, il existe une suite exacte:
\begin{equation}\label{suitebrauer}
0 \rightarrow \text{Br} \; K \rightarrow \bigoplus_{v \in \tilde{\mathcal{X}}^{(1)}} \text{Br} \; K_v \rightarrow \bigoplus_{w \in \tilde{\mathcal{X}}^{(2)}} \; \mathbb{Q}/\mathbb{Z} \rightarrow 0. \tag{$\star \star$}
\end{equation}
On obtient alors une suite exacte:
$$ \text{Br} \; K \rightarrow \bigoplus_{v \in X^{(1)}} \text{Br} \; K_v \rightarrow \Lambda \rightarrow 0,$$
où $\Lambda \cong \mathbb{Q}/\mathbb{Z}$ d'après le lemme \ref{Lambda}. Avec (i), cela montre alors que:
\begin{itemize}
\item[$\bullet$] si $U=X$, alors $H^3_c(X,\mathbb{G}_m) = H^3(X,\mathbb{G}_m)=\mathbb{Q}/\mathbb{Z}$;
\item[$\bullet$] si $U \neq X$, alors $\text{Br} \; K$ se surjecte sur $\bigoplus_{v \in U^{(1)}} \text{Br} \; K_v$ et donc $H^3(U,\mathbb{G}_m)=0$. Cela permet d'obtenir une suite exacte:
$$H^2(U,\mathbb{G}_m) \rightarrow \bigoplus_{v \in X \setminus U} \text{Br} \; K_v \rightarrow \mathbb{Q}/\mathbb{Z} \rightarrow 0.$$
Il suffit alors d'écrire la suite exacte de localisation (proposition 3.1.(1) de \cite{HS1}):
$$0 \rightarrow H^2_c(U,\mathbb{G}_m) \rightarrow H^2(U,\mathbb{G}_m) \rightarrow \bigoplus_{v \in X \setminus U} \text{Br}\; K_v\rightarrow  H^3_c(U,\mathbb{G}_m) \rightarrow 0$$
pour conclure que $H^3_c(U,\mathbb{G}_m) \cong \mathbb{Q}/\mathbb{Z}$.
\end{itemize}
\item[(iii)] D'après (i), on a une suite exacte:
$$0 \rightarrow H^2(X,\mathbb{G}_m) \rightarrow \text{Br}(K) \rightarrow \bigoplus_{v\in X^{(1)}} \text{Br}(K_v) \rightarrow H^3(X,\mathbb{G}_m) \rightarrow 0.$$
En exploitant la suite (\ref{suitebrauer}), on obtient alors une suite exacte:
$$0 \rightarrow \text{Br} \; X \rightarrow \bigoplus_{v \in \tilde{\mathcal{X}}^{(1)} \setminus X^{(1)}} \text{Br} \; K_v \rightarrow \bigoplus_{w \in \tilde{\mathcal{X}}^{(2)}} \; \mathbb{Q}/\mathbb{Z} \rightarrow H^3(X,\mathbb{G}_m) \rightarrow 0.$$
Donc, d'après le lemme \ref{Upsilon}, on a $\text{Br} \; X = \Upsilon \cong (\mathbb{Q}/\mathbb{Z})^{n_X}$.
\item[(iv)] Si $j: X \hookrightarrow \tilde{\mathcal{X}}$ désigne l'immersion ouverte et $\text{Div}_{\tilde{\mathcal{X}}}^{Y}$ le faisceau des diviseurs de $\tilde{\mathcal{X}}$ à support dans $Y=f^{-1}(\mathfrak{m})$ (c'est-à-dire le faisceau associé au préfaisceau qui à un morphisme étale $\varphi: T \rightarrow \tilde{\mathcal{X}}$ associe le groupe des diviseurs de $\tilde{\mathcal{X}}$ à support dans $\varphi^{-1}(Y)$), la suite exacte $0 \rightarrow \mathbb{G}_m \rightarrow j_* \mathbb{G}_m \rightarrow \text{Div}_{\tilde{\mathcal{X}}}^{Y} \rightarrow 0$ induit une suite exacte longue:
$$0 \rightarrow \mathcal{O}_{\tilde{\mathcal{X}}}(\tilde{\mathcal{X}})^{\times} \rightarrow \mathcal{O}_X(X)^{\times} \rightarrow \bigoplus_{v \in Y^{(0)}} \mathbb{Z} \rightarrow \text{Pic} \; \tilde{\mathcal{X}} \rightarrow \text{Pic} \; X \rightarrow 0.$$
Comme $\mathcal{O}_X(X) = \mathcal{O}_K$ est hensélien et la caractéristique résiduelle de $K$ est nulle, le lemme de Hensel montre que $\mathcal{O}_X(X)^{\times}$ est divisible. Par conséquent, le morphisme $\mathcal{O}_X(X)^{\times} \rightarrow \bigoplus_{v \in Y^{(0)}} \mathbb{Z}$ est nul, et on a une suite exacte:
\begin{equation}\label{Pic}
0 \rightarrow \bigoplus_{v \in Y^{(0)}} \mathbb{Z} \rightarrow \text{Pic} \; \tilde{\mathcal{X}} \rightarrow \text{Pic} \; X \rightarrow 0
\tag{$\dagger$} 
\end{equation}
et un isomorphisme:
$$\mathcal{O}_{\tilde{\mathcal{X}}}(\tilde{\mathcal{X}})^{\times} \cong \mathcal{O}_X(X)^{\times}.$$
Cela entraîne en particulier que le groupe $\mathcal{O}_{\tilde{\mathcal{X}}}(\tilde{\mathcal{X}})^{\times}$ est divisible. Comme le groupe de Brauer de $\tilde{\mathcal{X}}$ est trivial (d'après le corollaire 1.10 de \cite{COP}), la suite exacte de Kummer $1 \rightarrow \mu_m \rightarrow \mathbb{G}_m \rightarrow \mathbb{G}_m \rightarrow 1$ fournit des isomorphismes:
\begin{gather*}
{_m}\text{Pic} \;\tilde{\mathcal{X}} \cong H^1(\tilde{\mathcal{X}}, \mu_m) \cong H^1(Y,\mu_m) \cong {_m} \text{Pic} \; Y,\\
\text{Pic} \; \tilde{\mathcal{X}} \cong H^2(\tilde{\mathcal{X}},\mu_m) \cong H^1(Y,\mu_m) \cong (\text{Pic}\; Y) / m.
\end{gather*}
Nous devons donc calculer $\text{Pic}\; Y$. Comme cela a été rappelé au début de l'article, on a un isomorphisme: 
$$\text{Pic}\; Y \cong (k^{\times})^{t} \oplus \mathbb{Z}^{|Y^{(0)}|} \oplus (\mathbb{R}/\mathbb{Z})^{2\cdot \sum_{v \in Y^{(0)}} g_v}.$$
En notant $m_y$ le nombre de points de la normalisation de $Y$ qui sont au-dessus de $y$ pour chaque $y \in Y(k)$, l'entier $t$ est égal à $\mu - c + 1$ où $c$ est le nombre de composantes irréductibles de $Y$ et $\mu =\sum_{y \in Y(k) } (m_y - 1)$. Dans notre situation, $m_y = 2$ si $y$ est intersection de deux composantes de $Y$ et $m_y = 1$ sinon. Par conséquent, en termes du graphe $\Gamma$, on a $\mu = |V_2|$, $c=|V_1|$ et $t=|V_2|-|V_1|+1$. Comme les sommets de $\Gamma$ qui sont dans $V_2$ sont tous de degré 2, on a $|E| = 2|V_2|$, et donc la formule d'Euler pour les graphes montre que $t = |V_2|-|V_1|+1=-|V| + |E| + 1 =  c_{\Gamma}$. On obtient ainsi que:
\begin{equation*}\label{Picbis}
\text{Pic}\; Y \cong (k^{\times})^{c_{\Gamma}} \oplus \mathbb{Z}^{|Y^{(0)}|} \oplus (\mathbb{R}/\mathbb{Z})^{2\cdot \sum_{v \in Y^{(0)}} g_v}. \tag{$\dagger \dagger$}
\end{equation*}
On déduit alors de (\ref{Pic}) et (\ref{Picbis}) que, pour tout $m>0$, les groupes ${_m}\text{Pic}\; X$ et $\text{Pic} \; X /m$ sont finis, et que:
\begin{align*}
\frac{|{_m}\text{Pic}\; X|}{|\text{Pic} \; X /m|} & = m^{|Y^{(0)}|}\cdot\frac{|{_m}\text{Pic}\; \tilde{\mathcal{X}}|}{|\text{Pic} \; \tilde{\mathcal{X}} /m|} \;\;\;\;\;\;\;\;\;\;\;\;\;\;\;\;\;\;\; \text{(d'après (\ref{Pic}))}\\
& = m^{|Y^{(0)}|}\cdot\frac{|{_m}\text{Pic}\; Y|}{|(\text{Pic} \; Y) /m|}\\
& = m^{|Y^{(0)}|}\cdot m^{c_{\Gamma}-|Y^{(0)}| + 2 \sum_{v \in Y^{(0)}} g_v} \;\;\;\;\; \text{(d'après (\ref{Picbis}))}\\
& = m^{n_X}.
\end{align*}
\end{itemize}
\end{proof}

\hspace{4ex} Par conséquent, on obtient pour chaque entier $r \in \mathbb{Z}$ un accouplement:
$$ AV: \text{Ext}^r_U(F,\mathbb{G}_m) \times H^{3-r}_c(U,F) \rightarrow \mathbb{Q}/\mathbb{Z}.$$

Dans le paragraphe suivant, nous allons démontrer le théorème suivant:

\begin{theorem}\label{AVdim2}
L'accouplement $AV$ est un accouplement parfait de groupes finis.
\end{theorem}

Pour ce faire, on introduit le morphisme:
$$\alpha^r(U,F): \text{Ext}^r_U(F,\mathbb{G}_m) \rightarrow H^{3-r}_c(U,F)^D$$
induit par l'accouplement $AV$.

\subsection{Le cas d'un faisceau à support dans un fermé strict}\label{AVdébut}

Nous nous intéressons au cas où le support de $F$ est contenu dans un fermé strict de $U$.

\begin{lemma} \label{2.1}
Supposons que $F$ est à support dans un fermé strict $Z$ de $U$. Alors, pour tout $r\geq 0$, 
$$\bigoplus_{v \in Z} \text{Ext}^{r-1}_v(i_v^*F,\mathbb{Z}) \cong \text{Ext}^r_U(F,\mathbb{G}_m),$$
où, pour $v \in Z$, $i_v$ désigne l'immersion fermée $\text{Spec} \; k(v) \hookrightarrow U$.
\end{lemma}

\begin{proof}
Nous disposons d'une suite exacte de faisceaux sur $U$:
$$ 0 \rightarrow \mathbb{G}_m \rightarrow g_*\mathbb{G}_m \rightarrow \text{Div}_U \rightarrow 0$$
avec $\text{Div}_U = \bigoplus_{v \in U^{(1)}}  {i_v}_*\mathbb{Z}$, où $U^{(1)}$. On obtient alors une suite exacte longue:
$$ ... \rightarrow \text{Ext}^r_U(F,\mathbb{G}_m) \rightarrow \text{Ext}^r_U(F,g_*\mathbb{G}_m) \rightarrow \text{Ext}^r_U(F,\text{Div}_U) \rightarrow ... $$
\begin{itemize}
\item[$\bullet$] D'une part, comme dans la preuve de la proposition \ref{G_m}, pour $s >0$ le faisceau $R^sg_*\mathbb{G}_m$ est nul. Cela implique que la suite spectrale $\text{Ext}^r_U(F,R^sg_*\mathbb{G}_m) \Rightarrow \text{Ext}^{r+s}_{K}(g^*F,\mathbb{G}_m)$ induit des isomorphismes $\text{Ext}^r_U(F,g_*\mathbb{G}_m) \cong \text{Ext}^{r}_{K}(g^*F,\mathbb{G}_m)$. Or $g^*F=0$. Donc $\text{Ext}^r_U(F,g_*\mathbb{G}_m)=0$.
\item[$\bullet$] D'autre part: $$\text{Ext}^r_U(F,\text{Div}_U) = \bigoplus_{v \in U^{(1)}} \text{Ext}^r_U(F,{i_v}_*\mathbb{Z}) = \bigoplus_{v \in U^{(1)}} \text{Ext}^r_v({i_v}^*F,\mathbb{Z}) = \bigoplus_{v \in Z} \text{Ext}^r_v({i_v}^*F,\mathbb{Z}),$$
car $i_v^*F=0$ pour $v \in Z$.
\end{itemize}
Par conséquent, on obtient des isomorphismes $\bigoplus_{v \in Z} \text{Ext}^{r-1}_v(i_v^*F,\mathbb{Z}) \cong \text{Ext}^r_U(F,\mathbb{G}_m) $ pour $r \geq 0$.
\end{proof}

\begin{proposition}\label{suppfer}
Supposons que $F$ est à support dans un fermé strict $Z$ de $U$. Alors $\alpha^r(U,F)$ est un isomorphisme de groupes finis pour tout entier $r$.
\end{proposition}

\begin{proof}
Reprenons les notations de la démonstration précédente. Nous avons des isomorphismes:
\begin{align*}
\text{Ext}^r_U(F,\mathbb{G}_m) & \cong \bigoplus_{v \in Z} \text{Ext}^{r-1}_v(i_v^*F,\mathbb{Z}) \cong \bigoplus_{v \in Z} \text{Ext}^{r-1}_{k(v)}(F_{\overline{v}},\mathbb{Z}) \\ & \cong \bigoplus_{v \in Z} \text{Ext}^{r-2}_{k(v)}(F_{\overline{v}},\mathbb{Q}/\mathbb{Z}) \cong \bigoplus_{v \in Z} H^{r-2}(k(v),\text{Hom}(F_{\overline{v}},\mathbb{Q}/\mathbb{Z})), 
\end{align*}
$$ H^r_c(U,F) \cong H^r(Z,i^*F) \cong \bigoplus_{v \in Z} H^r(v,i_v^*F) \cong \bigoplus_{v \in Z} H^r(k(v),F_{\overline{v}}).$$
À travers ces isomorphismes, l'accouplement:
$$ \text{Ext}^r_U(F, \mathbb{G}_m) \times H_c^{3-r}(U,F) \rightarrow H_c^{3}(U,\mathbb{G}_m) \cong \mathbb{Q}/\mathbb{Z}$$
s'identifie à l'accouplement naturel:
$$ \bigoplus_{v \in Z} H^{r-2}(k(v),\text{Hom}(F_{\overline{v}},\mathbb{Q}/\mathbb{Z})) \times \bigoplus_{v \in Z} H^{3-r}(k(v),F_{\overline{v}})\rightarrow \mathbb{Q}/\mathbb{Z},$$
qui est bien un accouplement parfait de groupes finis.
\end{proof}

\subsection{Changement d'ouvert}

Nous allons maintenant voir que, si $V$ désigne un ouvert de $U$, alors $\alpha^r(U,F)$ est un isomorphisme pour tout entier $r$ si, et seulement si, $\alpha^r(V,F|_V)$ l'est. 

\begin{lemma}
Soit $0 \rightarrow F' \rightarrow F \rightarrow F'' \rightarrow 0$ une suite exacte de faisceaux constructibles sur $U$. Si $\alpha^r(U,F')$ et $\alpha^r(U,F'')$ sont des isomorphismes pour tout entier $r$, alors $\alpha^r(U,F)$ est un isomorphisme pour tout entier $r$.
\end{lemma}

\begin{proof}
On a un diagramme commutatif à lignes exactes:\\
\scriptsize{
\xymatrix{
 \text{Ext}^{r-1}_U(F',\mathbb{G}_m) \ar[r] \ar[d]^{\cong} & \text{Ext}^r_U(F'',\mathbb{G}_m)\ar[r] \ar[d]^{\cong} & \text{Ext}^r_U(F,\mathbb{G}_m)\ar[r] \ar[d] & \text{Ext}^r_U(F',\mathbb{G}_m) \ar[r] \ar[d]^{\cong} & \text{Ext}^{r+1}_U(F'',\mathbb{G}_m) \ar[d]^{\cong} \\
 (H^{4-r}_c(U,F'))^D \ar[r] & (H^{3-r}_c(U,F''))^D \ar[r] & (H^{3-r}_c(U,F))^D \ar[r] & (H^{3-r}_c(U,F'))^D \ar[r] & (H^{2-r}_c(U,F''))^D  
}}
\\
\normalsize
où les flèches verticales sont données par les $\alpha^r$. Le lemme des cinq permet de conclure.
\end{proof}

\begin{proposition}\label{change ouvert}
Soit $V$ un ouvert non vide de $U$. Alors $\alpha^r(U,F)$ est un isomorphisme pour tout entier $r$ si, et seulement si, $\alpha^r(V,F|_V)$ l'est.
\end{proposition}

\begin{proof}
Nous disposons d'une immersion ouverte $j: V \hookrightarrow U$ et d'une immersion fermée $i: U\setminus V \hookrightarrow U$, et donc d'une suite exacte de faisceaux sur $U$:
$$0 \rightarrow j_!j^*F \rightarrow F \rightarrow i_*i^*F \rightarrow 0.$$
D'après la proposition \ref{suppfer}, $\alpha^r(U,i_*i^*F)$ est un isomorphisme pour tout $r$, et donc, d'après le lemme précédent, $\alpha^r(U,F)$ est un isomorphisme pour tout $r$ si, et seulement si, $\alpha^r(U,j_!j^*F)$ l'est.\\
Or $\text{Ext}^r_U(j_!j^*F, \mathbb{G}_m) \cong \text{Ext}^r_V(j^*F,\mathbb{G}_m)$ et $H^r_c(U,j_!j^*F) \cong H^r_c(V,j^*F)$. À travers ces isomorphismes, $\alpha^r(U,j_!j^*F)$ s'identifie à $\alpha^r(V,j^*F)$. On en déduit immédiatement le résultat.
\end{proof}

\begin{remarque}\label{suppferrq}
La proposition \ref{suppfer} et la démonstration précédente montrent aussi que:
\begin{itemize}
\item[$\bullet$] $\text{Ext}_U^r(F,\mathbb{G}_m)$ est fini si, et seulement si, $\text{Ext}_V^r(F,\mathbb{G}_m)$ l'est;
\item[$\bullet$] pour $r\geq 4$, le groupe $\text{Ext}_U^r(F,\mathbb{G}_m)$ est nul si, et seulement si, $\text{Ext}_V^r(F,\mathbb{G}_m)$ l'est;
\item[$\bullet$] $H^r_c(U,F)$ est fini si, et seulement si, $H^r_c(V,F)$ l'est.
\end{itemize}
Par ailleurs, la suite exacte:
$$ \bigoplus_{v \in X \setminus U} H^{r-1}(K_v,F) \rightarrow H^r_c(U,F) \rightarrow H^r(U,F) \rightarrow \bigoplus_{v \in X \setminus U} H^r(K_v,F) $$
montre que $H^r(U,F)$ est fini si, et seulement si, $H^r_c(U,F)$ l'est.
\end{remarque}

\subsection{Propriétés de finitude}

Nous allons maintenant établir que $H^r_c(U,F)$ et $\text{Ext}^r_U(F,\mathbb{G}_m)$ sont bien finis.

\begin{lemma}\label{finiZ/m}
Soit $m \in \mathbb{N}^*$. Pour tout $r \in \mathbb{N}$, les groupes $H^r(U,\mathbb{Z}/m\mathbb{Z})$ et $\text{Ext}^r_U(\mathbb{Z}/m\mathbb{Z},\mathbb{G}_m)$ sont finis.
\end{lemma}

\begin{proof}
Les suites exactes de faisceaux:
\begin{gather*}
0 \rightarrow \mathbb{Z} \rightarrow \mathbb{Z} \rightarrow \mathbb{Z}/m\mathbb{Z} \rightarrow 0\\
0 \rightarrow \mu_m \rightarrow \mathbb{G}_m \rightarrow \mathbb{G}_m \rightarrow 0
\end{gather*}
fournissent des suites exactes longues:
\begin{gather*}
... \rightarrow \text{Ext}^r_X(\mathbb{Z}/m\mathbb{Z}, \mathbb{G}_m) \rightarrow H^r(X, \mathbb{G}_m) \rightarrow H^r(X,\mathbb{G}_m) \rightarrow ... \\
... \rightarrow H^r(X,\mathbb{Z}/m\mathbb{Z}) \rightarrow H^r(X,\mathbb{G}_m) \rightarrow H^r(X,\mathbb{G}_m) \rightarrow ...
\end{gather*}
Or, d'après le lemme \ref{G_m}, les groupes ${_m}H^r(X,\mathbb{G}_m)$ et $H^r(X,\mathbb{G}_m)/m$ sont finis pour tout $r \in \mathbb{N}$. On en déduit que $H^r(X,\mathbb{Z}/m\mathbb{Z})$ et $\text{Ext}^r_X(\mathbb{Z}/m\mathbb{Z},\mathbb{G}_m)$ sont finis. La remarque \ref{suppferrq} permet alors de conclure. 
\end{proof}

\begin{proposition} \label{finitude}
Rappelons que $U$ est un ouvert non vide de $X$ et que $F$ est un faisceau constructible sur $U$. Soit $r \in \mathbb{N}$.
\begin{itemize}
\item[(i)] Le groupe $H^r_c(U,F)$ est fini.
\item[(ii)] Le groupe $ \text{Ext}^r_U(F,\mathbb{G}_m)$ est fini.
\end{itemize}
\end{proposition}

\begin{proof}
\begin{itemize}
\item[(i)]
Soit $ V$ un ouvert non vide de $U$ sur lequel $F$ est localement constant. Considérons un morphisme fini étale $V' \rightarrow V$ tel que $V'$ est connexe et $F|_{V'}$ est constant. On remarque alors que $V'$ est la normalisation de $V$ dans une extension finie $L$ de $K$. Si l'on note $G= \text{Gal } (L/K)$, on dispose alors de la suite spectrale de Hochschild-Serre:
$$H^r(G,H^s(V',F|_{V'})) \Rightarrow H^{r+s}(V,F|_V).$$
Or $H^s(V',F|_{V'})$ est fini pour $s \geq 0$ d'après le lemme \ref{finiZ/m}. Par conséquent, $H^r(V,F|_V)$ est fini pour tout $r \geq 0$. La remarque \ref{suppferrq} permet alors de conclure. 
\item[(ii)] En reprenant les notations de (i), on a une suite spectrale:
$$H^r(G,\text{Ext}^s_{V'}(F|_{V'},\mathbb{G}_m)) \Rightarrow \text{Ext}^{r+s}_V(F|_V,\mathbb{G}_m).$$
Or $\text{Ext}^s_{V'}(F|_{V'},\mathbb{G}_m)$ est fini pour $s \geq 0$ d'après le lemme \ref{finiZ/m}. Par conséquent, $\text{Ext}^{r}_V(F|_V,\mathbb{G}_m)$ est fini pour tout $r \geq 0$. La remarque \ref{suppferrq} permet alors de conclure. 
\end{itemize}
\end{proof}

\subsection{Annulation de $H^r_c(U,F)$ pour $r\geq 4$}

\begin{lemma}\label{H=0}
On a $H^r(U,F) = 0$ pour $r \geq 4$.
\end{lemma}

\begin{proof}
Soit $V$ un ouvert de $U$ et écrivons la suite exacte de localisation:
$$... \rightarrow H^r(U,F) \rightarrow H^r(V,F) \rightarrow H^{r+1}_{U \setminus V}(U,F) \rightarrow ...$$
Le théorème d'excision montre que: $$H^{r}_{U \setminus V}(U,F) = \bigoplus_{v \in U \setminus V} H^{r}_v(\mathcal{O}^h_{v},F).$$
Or, pour $v \in U \setminus V$, on dispose de la suite exacte de localisation:
$$... \rightarrow  H^{r-1}(K_v^h,F) \rightarrow H^{r}_{v}(\mathcal{O}^h_{v},F) \rightarrow H^{r}(\mathcal{O}^h_{v},F) \rightarrow ...$$
Par dimension cohomologique, si $r \geq 4$, les groupes $H^{r-1}(K_v^h,F)$ et $H^{r}(\mathcal{O}^h_{v},F)$ sont nuls. Il en est donc de même du groupe $H^{r}_{v}(\mathcal{O}^h_{v},F)$. Par conséquent, pour $r \geq 4$, la restriction $H^r(U,F) \rightarrow H^r(V,F)$ est un isomorphisme. En notant $g: \text{Spec} \; K \hookrightarrow U$, on en déduit que:
$$H^r(U,F) \cong H^r(K,g^*F)=0$$
car $K$ est de dimension cohomologique 2.
\end{proof}

\begin{corollary}\label{Hc=0}
On a $H^r_c(U,F) = 0$ pour $r \geq 4$.
\end{corollary}

\begin{proof}
Il suffit d'appliquer le lemme précédent en remplaçant $U$ par $X$ et $F$ par $j_!F$.
\end{proof}

\subsection{Annulation de $\text{Ext}^r_U(F,\mathbb{G}_m)$ pour $r\geq 4$}

\begin{lemma}\label{faisceauext}
Supposons $F$ localement constant. Le faisceau $\text{\underline{Ext}}^r_U(F,\mathbb{G}_m)$ est de torsion pour $r=0$ et nul pour $r \geq 1$.
\end{lemma}

\begin{proof}
Calculons les tiges de $\text{\underline{Ext}}^r_U(F,\mathbb{G}_m)$:
\begin{itemize}
\item[$\bullet$] en $\overline{\eta}$, $\text{\underline{Ext}}^r_U(F,\mathbb{G}_m)_{\overline{\eta}}= \text{Ext}^r(F_{\overline{\eta}},{K^s}^{\times})$. Ce groupe est de $|F_{\overline{\eta}}|$-torsion pour $r=0$, et il est nul pour $r \geq 1$ puisque ${K^s}^{\times}$ est un groupe abélien divisible.
\item[$\bullet$] en $\overline{v}$ pour $v \in U$ différent de $\eta$, $\text{\underline{Ext}}^r_U(F,\mathbb{G}_m)_{\overline{v}}= \text{Ext}^r_{\mathbb{Z}}(F_{\overline{v}},{\mathcal{O}_v^{nr}}^{\times})$. Comme avant, ce groupe est de $|F_{\overline{\eta}}|$-torsion pour $r=0$, et il est nul pour $r \geq 1$ puisque ${\mathcal{O}_v^{nr}}^{\times}$ est un groupe abélien divisible.
\end{itemize}
Le lemme en découle immédiatement.
\end{proof}

\begin{lemma}\label{Ext=0}
On a $\text{Ext}^r_U(F,\mathbb{G}_m) = 0$ pour $r \geq 4$.
\end{lemma}

\begin{proof}
En utilisant la remarque \ref{suppferrq}, on peut supposer que $F$ est localement constant. Le lemme précédent montre alors que la suite spectrale $$H^r(U, \text{\underline{Ext}}^s_U(F,\mathbb{G}_m)) \Rightarrow \text{Ext}^{r+s}_U(F,\mathbb{G}_m)$$
dégénère en des isomorphismes $\text{Ext}^{r}_U(F,\mathbb{G}_m)\cong H^r(U, \text{\underline{Hom}}_U(F,\mathbb{G}_m)) $. Le faisceau $\text{\underline{Hom}}_U(F,\mathbb{G}_m)$ étant de torsion, il est limite inductive filtrée de faisceaux constructibles. On déduit alors du lemme \ref{H=0} que pour $r \geq 4$ le groupe $\text{Ext}^r_U(F,\mathbb{G}_m)$ est nul. 
\end{proof}

\subsection{Décomposition de $F$}

\begin{proposition}\label{dec}
Soit $V$ un ouvert non vide de $U$ tel que $F|_V$ est localement constant. Soit $L$ une extension finie de $K$ telle que la normalisation $ \pi_V: V_n \rightarrow V$ de $V$ dans $L$ est étale et $F|_{V_n}$ est constant. Soit $\pi_U: U_n \rightarrow U$ la normalisation de $U$ dans $L$. Soit $m>0$ tel que $mF=0$. Il existe alors un faisceau $F_n$ constructible constant sur $U_n$ de $m$-torsion, un faisceau $F_f$ constructible sur $U$ à support fini et un morphisme injectif $F \hookrightarrow {\pi_U}_*F_n \oplus F_f$.
\end{proposition}

\begin{proof}
On choisit pour $F_n$ le faisceau constant sur $U_n$ défini par le groupe abélien fini $F(V_n)$. Nous disposons alors d'un morphisme d'adjonction injectif $F|_V \hookrightarrow {\pi_V}_*\pi_V^*(F|_V) = {\pi_V}_*(F_n|_{V_n})$. Ce morphisme s'étend en un morphisme $F \rightarrow {\pi_U}_*F_n$. Le support du noyau de ce morphisme est contenu dans $U \setminus V$. Il suffit donc de trouver un faisceau $F_f$ constructible sur $U$ à support dans $U \setminus V$ et un morphisme $F \rightarrow F_f$ tel que le support du noyau soit contenu dans $V$. Si l'on note $i: U \setminus V \hookrightarrow U$, on choisit $F_f = i_*i^*F$, le morphisme $F \rightarrow F_f$ étant le morphisme d'adjonction.
\end{proof}

 \subsection{Comportement vis-à-vis de la normalisation}

Nous allons maintenant étudier le comportement du morphisme $\alpha^r$ vis-à-vis de la normalisation. Soit $L$ une extension finie galoisienne de $K$. Soit $U_n$ la normalisation de $U$ dans $L$ et notons $\pi: U_n \rightarrow U$. Considérons l'application norme $N_{U_n/U}: \pi_*\mathbb{G}_m  \rightarrow \mathbb{G}_m$ (voir par exemple le lemme II.3.9(a) de \cite{MilADT}).

\begin{lemma}
Soit $\mathcal{F}_n$ un faisceau constructible sur $U_n$. Pour tout $r\in \mathbb{Z}$, la composée $$N^r(\mathcal{F}_n): \text{Ext}^r_{U_n}(\mathcal{F}_n,\mathbb{G}_m) \rightarrow \text{Ext}^r_U(\pi_*\mathcal{F}_n,\pi_*\mathbb{G}_m) \rightarrow \text{Ext}^r_U(\pi_*\mathcal{F}_n,\mathbb{G}_m),$$
où le premier morphisme est induit par le foncteur exact $\pi_*$ et le second par la norme $N_{U_n/U}$, est un isomorphisme.
\end{lemma}

\begin{proof}
Soit $V_n$ un ouvert de $U_n$ sur lequel la restriction de $\pi$ est étale, et notons $j_n: V_n \hookrightarrow U_n$ l'immersion ouverte et $i_n: Z_n = U_n \setminus V_n \hookrightarrow U_n$ l'immersion fermée. Nous allons montrer que $N^r({j_n}_!j_n^*\mathcal{F}_n)$ et $N^r({i_n}_*i_n^*\mathcal{F}_n)$ sont des isomorphismes, ce qui impliquera que $N^r(\mathcal{F}_n)$ est un isomorphisme grâce à la suite exacte $0 \rightarrow {j_n}_!j_n^*\mathcal{F}_n \rightarrow \mathcal{F}_n \rightarrow {i_n}_*i_n^*\mathcal{F}_n \rightarrow 0$. 
\begin{itemize}
\item[$\bullet$] D'une part, comme $\pi \circ j_n$ est étale, on dispose d'isomorphismes: 
$$\text{Ext}^r_{V_{n}}(j_n^*\mathcal{F}_n,\mathbb{G}_m)=\text{Ext}^r_{V_n}(j_n^*\mathcal{F}_n,(\pi \circ j_n)^*\mathbb{G}_m)) \cong \text{Ext}^r_{U}(\pi_*{j_n}_!j_n^*\mathcal{F}_n,\mathbb{G}_m).$$
Or $\text{Ext}^r_{U_n}({j_n}_!j_n^*\mathcal{F}_n,\mathbb{G}_m) \cong \text{Ext}^r_{V_{n}}(j_n^*\mathcal{F}_n,\mathbb{G}_m)$. Cela montre immédiatement que $N^r({j_n}_!j_n^*\mathcal{F}_n)$ est un isomorphisme.
\item[$\bullet$] D'autre part, étant donné le choix de $V_n$, si l'on note $Z = \pi(Z_n)$, on remarque que $Z_n = \pi^{-1}(Z)= Z \times_U U_n$, et donc on dispose d'un diagramme commutatif:
\centerline{\xymatrix{
Z_n \ar[r]^{i_n} \ar[d]^{\pi_Z} & U_n \ar[d]^{\pi}\\
Z \ar[r]^{i} & U
}}
Ainsi, en utilisant le lemme \ref{2.1} et en tenant compte du fait que $\pi_Z$ est fini étale:
\begin{align*}
\text{Ext}^r_{U_n}({i_n}_*i_n^*\mathcal{F}_n,\mathbb{G}_m) &\cong \text{Ext}^{r-1}_{Z_n}(i_n^*\mathcal{F}_n,\mathbb{Z}) \cong \text{Ext}^{r-1}_{Z_n}(i_n^*\mathcal{F}_n,\pi_Z^*\mathbb{Z})\\
&\cong \text{Ext}^{r-1}_{Z}({\pi_Z}_*i_n^*\mathcal{F}_n,\mathbb{Z}) \cong \text{Ext}^r_{U}(\pi_*{i_n}_*i_n^*\mathcal{F}_n,\mathbb{G}_m),
\end{align*}
et donc $N^r({i_n}_*i_n^*\mathcal{F}_n)$ est un isomorphisme.
\end{itemize}
\end{proof}

\begin{proposition}\label{norm}
Soient $\mathcal{F}_n$ un faisceau constructible sur $U_n$ et $r \in \mathbb{Z}$. Alors $\alpha^r(U_n,\mathcal{F}_n)$ est un isomorphisme si, et seulement si, $\alpha^r(U,\pi_*\mathcal{F}_n)$ est un isomorphisme. 
\end{proposition}

\begin{proof}
Supposons que $U \neq X$. Comme $\pi$ est fini, on dispose d'un isomorphisme $H^3_c(U,\pi_*\mathbb{G}_m) \rightarrow H^3_c(U_n,\mathbb{G}_m)$. Notons $\text{Nm}_{U_n/U}$ la composée:
$$H^3_c(U_n,\mathbb{G}_m) \rightarrow H^3_c(U,\pi_*\mathbb{G}_m) \rightarrow H^3_c(U,\mathbb{G}_m)$$
le deuxième morphisme étant induit par $N_{U_n/U}$. Ce morphisme s'insère dans un diagramme cubique:\\
\centerline{\xymatrix{
& \mathbb{Q}/\mathbb{Z} \ar@{=}[rr] \ar@{=}'[d][dd] & & \mathbb{Q}/\mathbb{Z} \ar@{=}[dd] \\
\bigoplus_{w \in X_n \setminus U_n} \text{Br}(L_w) \ar[ru] \ar@{->>}[rr] \ar[dd]^{\text{Cores}} & & H^3_c(U_n,\mathbb{G}_m) \ar[ru]\ar[dd]& \\
& \mathbb{Q}/\mathbb{Z}  \ar@{=}'[r][rr] & & \mathbb{Q}/\mathbb{Z}  \\
\bigoplus_{v \in X \setminus U} \text{Br}(K_v) \ar[ru]\ar@{->>}[rr] & & H^3_c(U,\mathbb{G}_m)\ar[ru] & 
}}
où $X_n$ est la normalisation de $X$ dans $L$ et toutes les faces du cube, à l'exception de la face droite, sont commutatives. Cela entraîne que la face droite est aussi commutative. On en déduit un diagramme commutatif:\\
\centerline{\xymatrix{
\text{Ext}^r_{U_n}(\mathcal{F}_n,\mathbb{G}_m) \ar[d]_{\cong}^{N^r(\mathcal{F}_n)} \ar@{}[r]|{\times} & H^{3-r}_c(U_n,\mathcal{F}_n) \ar[r] & H^3_c(U_n,\mathbb{G}_m) \ar[r] \ar[d]_{\cong}^{\text{Nm}_{U_n/U}} & \mathbb{Q}/\mathbb{Z} \ar@{=}[d]\\
\text{Ext}^r_{U}(\pi_*\mathcal{F}_n,\mathbb{G}_m) \ar@{}[r]|{\times} &  H^{3-r}_c(U,\pi_*\mathcal{F}_n) \ar[r] \ar[u]_{\cong} & H^3_c(U,\mathbb{G}_m) \ar[r] & \mathbb{Q}/\mathbb{Z}
}}
et donc $\alpha^r(U_n,\mathcal{F}_n)$ est un isomorphisme si, et seulement si, $\alpha^r(U,\pi_*\mathcal{F}_n)$ est un isomorphisme.\\
Si $U = X$, on choisit $U'$ un ouvert strict de $X$. On a alors un diagramme cubique:
\centerline{\xymatrix{
& \mathbb{Q}/\mathbb{Z} \ar@{=}[rr] \ar@{=}'[d][dd] & & \mathbb{Q}/\mathbb{Z} \ar@{=}[dd] \\
H^3_c(U'_n,\mathbb{G}_m) \ar[ru] \ar[rr]^{\;\;\;\; \cong} \ar[dd]^{\text{Nm}_{U'_n/U'}} & & H^3_c(U_n,\mathbb{G}_m) \ar[ru]\ar[dd]& \\
& \mathbb{Q}/\mathbb{Z}  \ar@{=}'[r][rr] & & \mathbb{Q}/\mathbb{Z}  \\
H^3_c(U',\mathbb{G}_m) \ar[ru]\ar[rr]^{\cong} & & H^3_c(U,\mathbb{G}_m)\ar[ru] & 
}}
où toutes les faces, à l'exception de la face droite, sont commutatives. On en déduit la commutativité de la face droite. Il est maintenant possible de conclure exactement de la même manière que sous l'hypothèse $U \neq X$.

\end{proof}

\subsection{Propriété d'hérédité}

\begin{proposition} \label{rec}
Soit $r_0 \leq 3$ un entier. Supposons que $\alpha^r(X,\mathcal{F})$ soit un isomorphisme pour toute extension finie $K$ de $\mathbb{C}((x,t))$, pour tout faisceau constructible $\mathcal{F}$ sur $X$ et pour tout $r > r_0$. 
\begin{itemize}
\item[(i)] Supposons $r_0 \neq 3$. Pour toute extension finie $K$ de $\mathbb{C}((x,t))$ et pour tout faisceau constructible $\mathcal{F}$ sur $X$, le morphisme $\alpha^{r_0}(X,\mathcal{F})$ est surjectif.
\item[(ii)] On ne suppose plus que $r_0 \neq 3$. On suppose par contre que pour toute extension finie $K$ de $\mathbb{C}((x,t))$ et pour tout $m >0$, $\alpha^{r_0}(X,\mathbb{Z}/m\mathbb{Z})$ est un isomorphisme. Alors pour toute extension finie $K$ de $\mathbb{C}((x,t))$ et pour tout faisceau constructible $\mathcal{F}$ sur $X$, le morphisme $\alpha^{r_0}(X,\mathcal{F})$ est un isomorphisme.
\end{itemize} 
\end{proposition}

\begin{proof}
\begin{itemize}
\item[(i)] Soient $K$ une extension finie de $\mathbb{C}((x,t))$ et $\mathcal{F}$ un faisceau constructible sur $X$. Fixons $c \in H^{3-r_0}(X,\mathcal{F})$. Soit $\mathcal{I}$ un faisceau flasque de torsion muni d'une injection $\mathcal{F} \hookrightarrow \mathcal{I}$. Le faisceau $\mathcal{I}$ est la limite inductive de ses sous-faisceaux constructibles, parmi lesquels se trouve $\mathcal{F}$. Par conséquent, comme $H^{3-r_0}(X, \mathcal{I})=0$, il existe un sous-faisceau constructible $\mathcal{F}_0$ de $\mathcal{I}$ contenant $\mathcal{F}$ et tel que l'image de $c$ dans $H^{3-r_0}(X,\mathcal{F}_0)$ est nulle. Soit $\mathcal{F}_1$ le conoyau de l'injection $\mathcal{F} \hookrightarrow \mathcal{F}_0$. Nous obtenons alors un diagramme commutatif à lignes exactes:\\
\footnotesize{
\xymatrix{
 (\text{Ext}^{r_0+1}_X(\mathcal{F}_0,\mathbb{G}_m))^D \ar[r]  & (\text{Ext}^{r_0+1}_X(\mathcal{F}_1,\mathbb{G}_m))^D\ar[r]  & (\text{Ext}^{r_0}_X(\mathcal{F},\mathbb{G}_m))^D\ar[r]  & (\text{Ext}^{r_0}_X(\mathcal{F}_0,\mathbb{G}_m))^D   \\
 H^{2-r_0}(X,\mathcal{F}_0)\ar[u]^{\cong}_{(\alpha^{r_0+1}(X,\mathcal{F}_0))^D} \ar[r] & H^{2-r_0}(X,\mathcal{F}_1) \ar[r]\ar[u]^{\cong}_{(\alpha^{r_0+1}(X,\mathcal{F}_1))^D} & H^{3-r_0}(X,\mathcal{F}) \ar[r]\ar[u]_{(\alpha^{r_0}(X,\mathcal{F}))^D} & H^{3-r_0}(X,\mathcal{F}_0) \ar[u]_{(\alpha^{r_0}(X,\mathcal{F}_0))^D}
}}

\normalsize
Ainsi, si $(\alpha^{r_0}(X,\mathcal{F}))^D(c)=0$, alors $c = 0$. On en déduit que $(\alpha^{r_0}(X,\mathcal{F}))^D$ est injectif, d'où le résultat.

\item[(ii)] On reprend les mêmes notations qu'au début de (i). Soit $m>0$ tel que $m\mathcal{F}=0$. Comme $\mathcal{F}$ est constructible, on peut trouver un ouvert non vide $V$ de $X$ et un recouvrement étale fini connexe $V_n \rightarrow V$ tel que $\mathcal{F}|_{V_n}$ est constant. On remarque alors que $V_n$ est la normalisation de $V$ dans une extension finie $L$ de $K$. Notons $\pi_X: X_n \rightarrow X$ la normalisation de $X$ dans $L$. D'après la proposition \ref{dec}, il existe alors un faisceau $\mathcal{F}_n$ constructible constant sur $X_n$ de $m$-torsion, un faisceau $\mathcal{F}_f$ constructible sur $X$ à support fini et un morphisme injectif $\mathcal{F} \hookrightarrow {\pi_X}_*\mathcal{F}_n \oplus \mathcal{F}_f$. On note $\mathcal{F}_0 = {\pi_X}_*\mathcal{F}_n \oplus \mathcal{F}_f$. D'après la proposition \ref{suppfer}, on sait que $\alpha^r(X, \mathcal{F}_f)$ est un isomorphisme pour tout $r$. De plus, comme par hypothèse $\alpha^{r_0}(X_n,\mathbb{Z}/m'\mathbb{Z})$ est un isomorphisme pour $m'|m$, il en est de même de $\alpha^{r_0}(X_n,\mathcal{F}_n)$: la proposition \ref{norm} entraîne alors que $\alpha^{r_0}(X,{\pi_X}_*\mathcal{F}_n)$ est un isomorphisme. Par conséquent, $\alpha^{r_0}(X,\mathcal{F}_0)$ est un isomorphisme.\\
Notons $\mathcal{F}_1$ le conoyau de l'injection $\mathcal{F} \hookrightarrow \mathcal{F}_0$. Nous obtenons alors un diagramme commutatif à lignes exactes:\\

\scriptsize{
\xymatrix{
\text{Ext}^{r_0+1}_X(\mathcal{F}_0,\mathbb{G}_m) \ar[d]^{\cong}_{\alpha^{r_0+1}(X,\mathcal{F}_0)}  &\text{Ext}^{r_0+1}_X(\mathcal{F}_1,\mathbb{G}_m) \ar[l] \ar[d]^{\cong}_{\alpha^{r_0+1}(X,\mathcal{F}_1)}  & \text{Ext}^{r_0}_X(\mathcal{F},\mathbb{G}_m) \ar[l] \ar[d]_{\alpha^{r_0}(X,\mathcal{F})} & \text{Ext}^{r_0}_X(\mathcal{F}_0,\mathbb{G}_m) \ar[l]  \ar[d]_{\alpha^{r_0}(X,\mathcal{F}_0)}^{\cong} & \text{Ext}^{r_0}_X(\mathcal{F}_1,\mathbb{G}_m) \ar[l]  \ar[d]_{\alpha^{r_0}(X,\mathcal{F}_1)} \\
H^{2-r_0}(X,\mathcal{F}_0)^* &H^{2-r_0}(X,\mathcal{F}_1)^* \ar[l]& H^{3-r_0}(X,\mathcal{F})^* \ar[l]& H^{3-r_0}(X,\mathcal{F}_0)^*  \ar[l] & H^{3-r_0}(X,\mathcal{F}_1)^*  \ar[l]
}}
\normalsize
Par chasse au diagramme, $\alpha^{r_0}(X, \mathcal{F})$ est surjectif. Cela étant vrai pour tout faisceau constructible $\mathcal{F}$, $\alpha^{r_0}(X, \mathcal{F}_1)$ est aussi surjectif. Une nouvelle chasse au diagramme permet de conclure que $\alpha^{r_0}(X,\mathcal{F})$ est un isomorphisme.
\end{itemize}
\end{proof}

\subsection{Preuve de la dualité d'Artin-Verdier}\label{preuveAV}

Nous sommes à présent en mesure d'établir le théorème d'Artin-Verdier pour $K_0$. Pour ce faire, on prouve:

\begin{proposition}
Pour toute extension finie $K$ de $K_0$, pour tout faisceau constructible $F$ sur $X$ et pour tout $r_0 \in \mathbb{Z}$, le morphisme $\alpha^{r_0}(X,F)$ est un isomorphisme.
\end{proposition}

\begin{proof} On procède par étapes:
\begin{itemize} 
\item[\textbf{A.}] \underline{Pour $r_0 >3$:} \\
Le lemme \ref{Ext=0} entraîne immédiatement que $\alpha^{r_0}(X, F)$ est un isomorphisme.
\item[\textbf{B.}] \underline{Pour $r_0 =3$:}\\
Soit $K$ une extension finie de $K_0$. Soit $m >0$ un entier. La suite exacte courte de faisceaux $ 0 \rightarrow \mathbb{Z} \rightarrow \mathbb{Z} \rightarrow \mathbb{Z}/m\mathbb{Z} \rightarrow 0$ induit une suite exacte longue de cohomologie:
\begin{equation}\label{suiteext}
... \rightarrow \text{Ext}^r_X(\mathbb{Z}/m\mathbb{Z}, \mathbb{G}_m) \rightarrow H^r(X, \mathbb{G}_m) \rightarrow H^r(X,\mathbb{G}_m) \rightarrow ... \tag{$\ddagger$}
\end{equation}
Comme $\text{Br}\; X$ est divisible et $H^3(X, \mathbb{G}_m) = \mathbb{Q}/\mathbb{Z}$ d'après le lemme \ref{G_m}, on déduit que $\text{Ext}^3_X(\mathbb{Z}/m\mathbb{Z}, \mathbb{G}_m)\cong \frac{1}{m}\mathbb{Z}/\mathbb{Z}$. D'autre part, $H^0(X, \mathbb{Z}/m\mathbb{Z}) = \mathbb{Z}/m\mathbb{Z}$. Vu à travers ces isomorphismes, l'accouplement d'Artin-Verdier devient $(x, y) \mapsto x \tilde{y}$ pour $x \in \frac{1}{m}\mathbb{Z}/\mathbb{Z}$, $y \in \mathbb{Z}/m\mathbb{Z}$ et $\tilde{y} \in \mathbb{Z}$ un relèvement quelconque de $y$. Cela impose que $\alpha^{3}(X, \mathbb{Z}/m\mathbb{Z})$ est un isomorphisme. La proposition \ref{rec}(ii) entraîne alors immédiatement que $\alpha^{3}(X,F)$ est un isomorphisme pour tout faisceau constructible $F$ sur $X$.
\item[\textbf{C.}] \underline{Pour $r_0 =2$:}\\
Soit $K$ une extension finie de $K_0$. Soit $m >0$ un entier. D'après la proposition \ref{rec}(i), $\alpha^2(X, \mathbb{Z}/m\mathbb{Z})$ est surjectif. Pour voir que c'est un isomorphisme, il suffit de montrer que $H^1(X, \mathbb{Z}/m\mathbb{Z})$ et $\text{Ext}^2_X(\mathbb{Z}/m\mathbb{Z}, \mathbb{G}_m)$ ont même cardinal.\\
La suite exacte (\ref{suiteext}) impose que $$|\text{Ext}^2_X(\mathbb{Z}/m\mathbb{Z}, \mathbb{G}_m)| = |\text{Pic}(X)/m\text{Pic}(X)| |{_m}\text{Br}(X)|.$$ D'autre part, la suite exacte de Kummer montre que:
$$|H^1(X, \mathbb{Z}/m\mathbb{Z})| = |\mathcal{O}_K^{\times}/\mathcal{O}_K^{\times, m}| |_{m}\text{Pic}(X)|= |_{m}\text{Pic}(X)|$$
car $\mathcal{O}_K$ est divisible. Donc, en utilisant le lemme \ref{G_m}:
$$\frac{|H^1(X, \mathbb{Z}/m\mathbb{Z})|}{|\text{Ext}^2_X(\mathbb{Z}/m\mathbb{Z}, \mathbb{G}_m)|} = \frac{|_{m}\text{Pic}(X)|}{|\text{Pic}(X)/m\text{Pic}(X)|} \cdot |{_m}\text{Br}(X)|^{-1}=\frac{m^{n_X}}{m^{n_X}}=1.$$
Par conséquent, $\alpha^2(X,\mathbb{Z}/m\mathbb{Z})$ est un isomorphisme. La proposition \ref{rec}(ii) permet alors de déduire que $\alpha^2(X, F)$ est un isomorphisme pour tout faisceau constructible $F$ sur $X$.
\item[\textbf{D.}] \underline{Pour $r_0 =1$:}\\
Soit $K$ une extension finie de $K_0$. Soit $m >0$ un entier. D'après la proposition \ref{rec}(i), $\alpha^1(X, \mathbb{Z}/m\mathbb{Z})$ est surjective. Pour voir que c'est un isomorphisme, il suffit de montrer que $H^2(X, \mathbb{Z}/m\mathbb{Z})$ et $\text{Ext}^1_X(\mathbb{Z}/m\mathbb{Z}, \mathbb{G}_m)$ ont même cardinal.\\
La suite exacte (\ref{suiteext}) impose que $$|\text{Ext}^1_X(\mathbb{Z}/m\mathbb{Z}, \mathbb{G}_m)| = |\mathcal{O}_K^{\times}/\mathcal{O}_K^{\times, m}||{_m}\text{Pic}(X)| = |{_m}\text{Pic}(X)|.$$ D'autre part, la suite exacte de Kummer montre que:
$$|H^2(X, \mathbb{Z}/m\mathbb{Z})| = |\text{Pic}(X)/m\text{Pic}(X)| |_{m}\text{Br}(X)|.$$
Donc, en utilisant le lemme \ref{G_m}:
$$\frac{|H^2(X, \mathbb{Z}/m\mathbb{Z})|}{|\text{Ext}^1_X(\mathbb{Z}/m\mathbb{Z}, \mathbb{G}_m)|} = \frac{|\text{Pic}(X)/m\text{Pic}(X)|}{|_{m}\text{Pic}(X)|} \cdot |{_m}\text{Br}(X)|=\frac{m^{n_X}}{m^{n_X}}=1.$$
Par conséquent, $\alpha^1(X,\mathbb{Z}/m\mathbb{Z})$ est un isomorphisme. La proposition \ref{rec}(ii) permet alors de déduire que $\alpha^1(X, F)$ est un isomorphisme pour tout faisceau constructible $F$ sur $X$.
\item[\textbf{E.}] \underline{Pour $r_0 =0$:}\\
Soit $K$ une extension finie de $K_0$. Soit $m >0$ un entier. 
D'après la proposition \ref{rec}(i), $\alpha^0(X, \mathbb{Z}/m\mathbb{Z})$ est surjective. Comme $\text{Br} \; X$ est divisible et $H^3(X,\mathbb{G}_m) \cong \mathbb{Q}/\mathbb{Z}$ d'après le lemme \ref{G_m}, la suite de Kummer montre que:
$$|H^3(X, \mathbb{Z}/m\mathbb{Z})|=m= |\text{Hom}_X(\mathbb{Z}/m\mathbb{Z}, \mathbb{G}_m)|.$$
Par conséquent, $\alpha^0(X,\mathbb{Z}/m\mathbb{Z})$ est un isomorphisme. La proposition \ref{rec}(ii) permet alors de déduire que $\alpha^0(X, F)$ est un isomorphisme pour tout faisceau constructible $F$ sur $X$.
\item[\textbf{F.}] \underline{Pour $r_0 <0$:}\\
Le corollaire \ref{Hc=0} entraîne immédiatement que $\alpha^{r_0}(X, F)$ est isomorphisme puisque $H^{3-r_0}_c(X,F)$ est nul.
\end{itemize}
\end{proof}

Les propositions \ref{change ouvert} et \ref{finitude} permettent alors d'obtenir le théorème \ref{AVdim2}. On en déduit le corollaire:
\\~
\begin{corollary}\label{AVdim2bis}
Soit $F$ un faisceau constructible localement constant sur $U$. Notons $F' = \underline{Hom}_U(F, \mathbb{G}_m)$. Il existe alors un accouplement parfait de groupes finis:
$$H^r(U,F') \times H^{3-r}_c(U,F) \rightarrow H^3_c(U,\mathbb{G}_m) = \mathbb{Q}/\mathbb{Z}.$$
\end{corollary}

\begin{proof}
Il suffit d'identifier $H^r(U,F')$ et $\text{Ext}^{r}_U(F,\mathbb{G}_m)$. Cela a été fait dans le lemme \ref{Ext=0}. 
\end{proof}

\begin{remarque}\label{sép}
Dans le cas où $k$ est séparablement clos de caractéristique $p>0$, le corollaire précédent reste vrai à condition de supposer que $F$ est de $n$-torsion avec $n$ non divisible par $p$. On attire en particulier l'attention sur le fait que, si $k$ est séparablement clos mais non algébriquement clos, le corps $k(t)$ est encore de dimension cohomologique 1 puisque la cohomologie étale est invariante par extension purement inséparable.
\end{remarque}

\section{Théorèmes de dualité arithmétique}\label{secdualdim2}

\hspace{4ex} On ne suppose plus $k$ séparablement clos et on note $\mathbb{P}$ l'ensemble des nombres premiers. Dans toute la suite de l'article, on supposera qu'il existe des entiers $d \geq -1$ et $p \in \{0\} \cup \mathbb{P}$ tels que:
\begin{itemize}
\item[$\bullet$] $\text{Car}(k) \in \{0,p\}$;
\item[$\bullet$] le corps $k$ est de dimension cohomologique $d+1$;
\item[$\bullet$] pour chaque entier naturel non nul $n$, on a un isomorphisme $H^{d+1}(k,\mu_n^{\otimes d}) \cong \mathbb{Z}/n\mathbb{Z}$ de sorte que, pour $m|n$, le morphisme naturel $H^{d+1}(k,\mu_m^{\otimes d})\rightarrow H^{d+1}(k,\mu_n^{\otimes d})$ s'identifie au morphisme qui envoie la classe de 1 dans $\mathbb{Z}/m\mathbb{Z}$ sur la classe de $n/m$ dans $\mathbb{Z}/n\mathbb{Z}$;
\item[$\bullet$] pour tout $\text{Gal}(k^s/k)$-module fini $M$ d'ordre $n$ non divisible par $p$, le cup-produit induit un accouplement parfait de groupes finis: $$H^r(k,M) \times H^{d+1-r}(k,\text{Hom}(M,\mu_n^{\otimes d})) \rightarrow H^{d+1}(k,\mu_n^{\otimes d}) \cong \mathbb{Z}/n\mathbb{Z}.$$
\end{itemize}

On dira que $k$ est \textbf{un corps $(p,d)$-bon pour la dualité}.

\begin{example}
 Un corps séparablement clos de caractéristique $p$ est $(p,-1)$-bon pour la dualité. Un corps fini de caratéristique $p$ est $(p,0)$-bon pour la dualité d'après l'exemple I.1.10 de \cite{MilADT}. Un corps $\ell$-adique est $(0,1)$-bon pour la dualité d'après le corollaire I.2.3 de \cite{MilADT}. En procédant comme dans la preuve du théorème I.2.17 de \cite{MilADT}, on peut montrer qu'un corps de valuation discrète complet de corps résiduel $(p,d)$-bon pour la dualité est $(p,d+1)$-bon pour la dualité. Par exemple, les corps $d$-locaux au sens de \cite{Izq1} sont $(p,d)$-bons pour la dualité si l'on choisit pour $p$ la caractéristique du corps 1-local correspondant. 
\end{example}

\hspace{4ex} On rappelle que $R_0$ désigne une $k$-algèbre commutative, locale, géométriquement intègre (ie $R_0 \otimes_k k^s$ est intègre), normale, hensélienne, excellente, de dimension 2 et de corps résiduel $k$. On note toujours $K_0$ son corps des fractions, $\mathfrak{m}_0$ son idéal maximal, $\mathcal{X}_0$ le schéma $ \text{Spec} \; R_0$ et $X_0$ le schéma $\mathcal{X}_0 \setminus \{\mathfrak{m}_0\}$. Le but de cette section est d'établir des théorèmes de dualité de type Poitou-Tate pour les modules finis et les tores sur $K_0$.

\begin{remarque}
En comparant tous les résultats qui vont suivre à ceux des articles \cite{HS1}, \cite{HS2}, \cite{CTH}, \cite{Izq1} et \cite{Izq3}, on remarquera que le corps $K_0$ se comporte de manière très similaire au corps des fonctions d'une courbe sur un corps $(p,d+1)$-bon pour la dualité. Par exemple, en ignorant les phénomènes liés à la caractéristique positive, $\mathbb{C}((x,y))$ se comporte comme $\mathbb{C}((x))(y)$ et $\mathbb{F}_p(t)$, et $\mathbb{F}_p((x,y))$ et $\mathbb{C}((t))((x,y))$ se comportent comme $\mathbb{Q}_p(x)$ et $\mathbb{C}((t))((x))(y)$. 
\end{remarque}

\subsection{Dualité d'Artin-Verdier et descente galoisienne}

\begin{lemma}
La $k^s$-algèbre $\overline{R_0}=R_0 \otimes_k k^s$ est locale, intègre, normale, hensélienne, excellente, de dimension 2, de corps résiduel $k^s$.
\end{lemma}

\begin{proof}
Comme $R_0$ est géométriquement intègre, $\overline{R_0}$ est intègre. Ainsi, pour chaque extension finie $l$ de $k$, la $l$-algèbre $R_0 \otimes_k l$ est intègre. De plus, c'est une extension finie étale de $R_0$. Par conséquent, comme $R_0$ est local normal hensélien excellent, il en est de même de $R_0 \otimes_k l$. On vérifie immédiatement que l'idéal maximal de $R_0 \otimes_k l$ est $\mathfrak{m}_l = \mathfrak{m}_0 \otimes_k l$ et que le corps résiduel de $R_0 \otimes _k l$ est $l$. On déduit alors de la proposition 1 de l'appendice de \cite{Bou} et du corollaire 4.4 de \cite{Mar} que $\overline{R_0}$ est un anneau local normal hensélien excellent d'idéal maximal $\overline{\mathfrak{m}} = \mathfrak{m}_0 \otimes_k k^s$ et de corps résiduel $k^s$. Finalement, l'injection $R_0 \hookrightarrow \overline{R}_0$ étant plate, le théorème 4.3.12 de \cite{Liu} montre que $\dim \; \overline{R_0} = \dim \; R_0 = 2$.
\end{proof}

\begin{theorem}\label{AVdquelc}
Soient $U$ un ouvert de $X_0$ et $F$ un schéma en groupes fini étale abélien sur $U$ de $n$-torsion, avec $n$ non divisible par $p$. Notons $F' = \text{Hom}(F, \mu_n^{\otimes (d+2)})$. Le groupe $H^{d+4}_c(U,\mu_n^{\otimes (d+2)})$ est isomorphe à $\mathbb{Z}/n\mathbb{Z}$ et l'accouplement naturel:
$$H^r(U,F) \times H^{d+4-r}_c(U,F') \rightarrow H^{d+4}_c(U,\mu_n^{\otimes (d+2)}) \cong \mathbb{Z}/n\mathbb{Z}$$
est un accouplement parfait de groupes finis.
\end{theorem}

\begin{proof}
Notons $G$ le groupe de Galois absolu de $k$. Comme $k$ est un corps $(p,d)$-bon pour la dualité, on montre exactement comme dans le théorème 1.1 de \cite{Izq1} que l'on a un isomorphisme dans la catégorie dérivée:
$$\mathbb{R}\text{Hom}_{G,\mathbb{Z}/n\mathbb{Z}} (M^{\bullet}, \mathbb{Z}/n\mathbb{Z}(d))[d+1] \cong \mathbb{R}\text{Hom}_{\mathbb{Z}/n\mathbb{Z}} (\mathbb{R}\Gamma_G M^{\bullet}, \mathbb{Z}/n\mathbb{Z})$$
pour chaque complexe $M^{\bullet}$ de $G$-modules discrets de $n$-torsion borné inférieurement. Par ailleurs, d'après le théorème \ref{AVdim2bis} et le lemme précédent, si l'on note $\overline{U} = U \times_k k^s$ et $\overline{F}$ la restriction de $F$ à $\overline{U}$, on a des accouplements parfaits $G$-équivariants de groupes finis:
$$H^r(\overline{U},\overline{F}) \times H^{3-r}_c(\overline{U},\text{Hom}(\overline{F},\mu_n)) \rightarrow \mathbb{Z}/n\mathbb{Z}(-1).$$
Comme dans le théorème 1.1 de \cite{Izq1}, cela fournit un isomorphisme dans la catégorie dérivée des groupes abéliens:
$$\mathbb{R} \text{Hom}_{\overline{U},\mathbb{Z}/n\mathbb{Z}}(\overline{F},\mu_n) \cong \mathbb{R} \text{Hom}_{\mathbb{Z}/n\mathbb{Z}} (\mathbb{R}\Gamma_c(\overline{U},\overline{F}), \mathbb{Z}/n\mathbb{Z}(-1))[-3].$$
On conclut alors exactement de la même manière que la proposition 2.1 de \cite{Izq1}.
\end{proof}

\subsection{Dualité de Poitou-Tate pour les modules finis}

\hspace{4ex} Soit $F$ un $\text{Gal}(K_0^s/K_0)$-module fini d'ordre $n$ non divisible par $p$. Notons $F'=\text{Hom}(F,\mu_n^{\otimes (d+2)})$, et posons, pour $r \in \{1,2,...,d+3\}$:
$$\Sha^r(K_0,F)=\text{Ker}\left( H^r(K,F) \rightarrow \prod_{v \in X_0^{(1)}} H^r(K_{0,v},F)\right).$$ 

\begin{theorem}\label{PTC((x,t))}
Pour $r \in \{1,2,...,d+3\}$, il existe un accouplement parfait de groupes finis:
$$\Sha^r(K_0,F) \times \Sha^{d+4-r}(K_0,F') \rightarrow \mathbb{Q}/\mathbb{Z}.$$
\end{theorem}

\begin{remarque}
Plaçons-nous dans le cas $k=\mathbb{C}$, $p=0$, $d=-1$. Bien sûr, on aurait pu définir les groupes de Tate-Shafarevich en tenant compte de toutes les valuations discrètes de rang 1 de $K_0$:
$$\Sha^r_{tot}(K_0,F):=\text{Ker}\left( H^r(K_0,F) \rightarrow \prod_{v \in \Omega_{K_0}} H^r(K_{0,v},F)\right),$$ 
où $\Omega_{K_0}$ désigne toutes les valuations discrètes de rang 1 de $K_0$. Cependant, on ne peut pas espérer une dualité parfaite de la forme:
$$\Sha^1_{tot}(K_0,F) \times \Sha^{2}_{tot}(K_0,F') \rightarrow \mathbb{Q}/\mathbb{Z}$$
puisque $\Sha^2_{tot}(K_0,\mu_n)$ est toujours nul (théorème 1.2 de \cite{Hu}) alors que $\Sha^1_{tot}(K_0,\mathbb{Z}/n\mathbb{Z})$ peut être non nul (théorème 5 de \cite{Jaw}). Par ailleurs, ce dernier point montre que dans le théorème \ref{PTC((x,t))}, les groupes de Tate-Shafarevich qui apparaissent ne sont pas toujours nuls.
\end{remarque}

\hspace{4ex} Nous allons maintenant brièvement expliquer comment on obtient le théorème \ref{PTC((x,t))}, la preuve étant analogue à celle du théorème 2.4 de \cite{Izq1}. Soit $U$ un ouvert non vide de $X_0$. Considérons un schéma en groupes fini étale sur $U$ prolongeant $F$: on le notera toujours $F$. Posons:
\begin{align*}
D^r_{sh}(U,F) = \text{Ker}\left(  H^r(U,F) \rightarrow \prod_{v \in X^{(1)}_0} H^r(K_v,F)\right), \\
\mathcal{D}^{d+4-r}(U,F'):= \text{Im}\left(  H^{d+4-r}_c(U,F')  \rightarrow H^{d+4-r} (K_0, F')\right) .
\end{align*}

\begin{proposition} 
\begin{itemize}
\item[(i)] Il existe un ouvert non vide $U_0$ de $U$ tel que $\mathcal{D}^{d+4-r}(V,F') =  \Sha^{d+4-r}(K_0, F')$ pour chaque ouvert non vide $V \subseteq U_0$.
\item[(ii)] La suite $ \bigoplus_{v \in X_0^{(1)}} H^{d+3-r}(K_{0,v},F') \rightarrow H^{d+4-r}_c(U,F') \rightarrow \mathcal{D}^{d+4-r}(U,F') \rightarrow 0$ est exacte.
\end{itemize}
\end{proposition}

\begin{proof}
\begin{itemize}
\item[(i)] Si $V \subseteq V'$ sont des ouverts non vides de $U$, on a $\mathcal{D}^{d+4-r}(V,F') \subseteq \mathcal{D}^{d+4-r}(V',F')$. Comme $\mathcal{D}^{d+4-r}(U,F')$ est fini, il existe un ouvert non vide $U_0$ de $U$ tel que $\mathcal{D}^{d+4-r}(V,F')=\mathcal{D}^{d+4-r}(U_0,F')$ pour chaque ouvert non vide $V \subseteq U_0$. Par ailleurs, pour $V \subseteq V'$ des ouverts non vides de $U$, on a un diagramme commutatif à lignes exactes:\\
\centerline{\xymatrix{
    \bigoplus_{v \not\in V'} H^{d+3-r}(K_{0,v}, F')  \ar[r] \ar[d]  & H^{d+4-r}_c(V',F') \ar[r] & H^{d+4-r}(V',F') \ar[d] \ar[r] & \bigoplus_{v \not\in V'} H^{d+4-r}(K_{0,v}, F')  \\
    \bigoplus_{v \not\in V} H^{d+3-r}(K_{0,v}, F')  \ar[r]  & H^{d+4-r}_c(V,F') \ar[u] \ar[r] & H^{d+4-r}(V,F')\ar[r] & \bigoplus_{v \not\in V} H^{d+4-r}(K_{0,v}, F').\ar[u]
  }
}
En prenant $V'=U_0$, une simple chasse au diagramme permet de montrer que $U_0$ convient.
\item[(ii)] Ce n'est qu'une chasse au diagramme dans le diagramme de (i).
\end{itemize}
\end{proof}

On a alors un diagramme commutatif à lignes exactes:\\
\centerline{\xymatrix{
0 \ar[r] & D^r_{sh}(U, F) \ar[r]\ar[d] & H^r(U,F) \ar[r]\ar[d]^{\cong} & \prod_{v \in X^{(1)}_0} H^r(K_{0,v},F)\ar[d]^{\cong}\\
0 \ar[r] & \mathcal{D}^{d+4-r}(U,F')^D \ar[r] & H^{d+4-r}_c(U,F')^D \ar[r] & \prod_{v \in X^{(1)}_0} H^{d+3-r}(K_{0,v},F')^D.
}}
Le morphisme vertical central est un isomorphisme d'après le théorème \ref{AVdim2bis}. Pour voir que le morphisme vertical de droite est un isomorphisme, il suffit d'appliquer le même méthode que Milne dans le théorème I.2.17 de \cite{MilADT} et d'utiliser que $k$ est un corps $(p,d)$-bon pour la dualité. On obtient alors un isomorphisme $D^r_{sh}(U, F) \cong \mathcal{D}^{d+4-r}(U,F')^D$. Il suffit de passer à la limite sur $U$ pour obtenir le théorème \ref{PTC((x,t))}.

\subsection{Dualité de Poitou-Tate pour les tores}

\begin{theorem}\label{PTC((x,t))tore}
Soit $T$ un $K_0$-tore de module de caractères $\hat{T}$ et de tore dual $T'$.
\begin{itemize}
\item[(i)] Supposons $d=-1$ (par exemple $k$ séparablement clos). On a des accouplements parfaits de groupes finis:
\begin{gather*}
\Sha^1(K_0,\hat{T})_{\text{non}-p} \times \overline{\Sha^2(K_0,T)}_{\text{non}-p} \rightarrow \mathbb{Q}/\mathbb{Z}\\
\Sha^1(K_0,T)_{\text{non}-p} \times \overline{\Sha^2(K_0,\hat{T})}_{\text{non}-p} \rightarrow \mathbb{Q}/\mathbb{Z}.
\end{gather*}
\item[(ii)] Supposons que $d=0$ (par exemple $k = \mathbb{C}((t))$ ou $k$ fini). On a un accouplement parfait de groupes finis:
$$\Sha^1(K_0,T)_{\text{non}-p} \times \overline{\Sha^2(K_0,T')}_{\text{non}-p} \rightarrow \mathbb{Q}/\mathbb{Z}.$$
\end{itemize}
\end{theorem}

\begin{proof}
Les preuves sont semblables à celle du théorème \ref{PTC((x,t))}. 
\begin{itemize}
\item[(i)] La preuve est en tous points analogue à celle des théorèmes 7.1 et 7.2 de \cite{CTH}. Rappelons brièvement les grandes étapes de la preuve du premier accouplement, le second étant obtenu de manière très similaire. Pour simplifier, on suppose que $p=0$.\\
Soit $\ell$ un nombre premier. On considère $\mathcal{T}$ un tore étendant $T$ sur un ouvert non vide $U$ de $X_0$. Pour $v \in X_0^{(1)}$, un argument identique à celui de la proposition 3.4 de \cite{CTH} montre que l'accouplement naturel:
$$H^1(K_{0,v},\hat{T}) \times H^1(K_{0,v},T) \rightarrow H^2(K_{0,v},\mathbb{G}_m) \cong \mathbb{Q}/\mathbb{Z}$$
est un accouplement parfait de groupes finis. De plus, le théorème \ref{AVdquelc} permet de prouver que, pour $V$ ouvert non vide de $U$, l'accouplement naturel $$H^1(V,\hat{\mathcal{T}}) \times H^2_c(V,\mathcal{T}) \rightarrow H^3_c(V,\mathbb{G}_m) \cong \mathbb{Q}/\mathbb{Z}$$ induit un accouplement parfait de groupes finis:
$$H^1(V,\hat{\mathcal{T}})^{(\ell)}\{\ell\} \times H^2_c(V,\mathcal{T})\{\ell\}^{(\ell)} \rightarrow \mathbb{Q}/\mathbb{Z}.$$
Par conséquent, en introduisant pour chaque ouvert non vide $V$ de $U$ les groupes:
\begin{gather*}
D^1_{sh}(V,\hat{\mathcal{T}}) = \text{Ker}\left( H^1(V,\hat{\mathcal{T}}) \rightarrow \prod_{v \in X_0^{(1)}} H^1(K_{0,v},\hat{T})\right),\\
\mathcal{D}^2(V,\mathcal{T}) = \text{Im}\left( H^2_c(V,\mathcal{T}) \rightarrow H^2(K_0,T)\right),
\end{gather*}
et en utilisant la suite exacte de localisation (proposition 3.1.(1) de \cite{HS1}), on obtient un diagramme commutatif à lignes exactes:\\
\xymatrix{
0 \ar[r] & D^1_{sh}(V, \hat{\mathcal{T}})\{\ell\} \ar[r]\ar[d] & H^1(V,\hat{\mathcal{T}})\{\ell\} \ar[r]\ar[d]^{\cong} & \prod_{v \in X^{(1)}_0} H^1(K_{0,v},\hat{\mathcal{T}})\{\ell\}\ar[d]^{\cong}\\
0 \ar[r] & \overline{\mathcal{D}^{2}(V,\mathcal{T})\{\ell\}}^D \ar[r] & \overline{H^{2}_c(V,\mathcal{T})\{\ell\}}^D \ar[r] & \prod_{v \in X^{(1)}_0} H^{1}(K_{0,v},\mathcal{T})\{\ell\}^D.
}\\
Cela montre que $D^1_{sh}(V, \hat{\mathcal{T}})\{\ell\}$ et $\overline{\mathcal{D}^{2}(V,\mathcal{T})\{\ell\}}$ sont duaux l'un de l'autre. Reste à identifier ces deux groupes à des groupes de Tate-Shafarevich:
\begin{itemize}
\item[$\bullet$] on montre aisément (comme dans le lemme 7.3 de \cite{CTH}) que le morphisme $H^1(V,\hat{\mathcal{T}}) \rightarrow H^1(K_0,\hat{T})$ est injectif et donc que $D^1_{sh}(V, \hat{\mathcal{T}}) = \Sha^1(K_0,\hat{T})$;
\item[$\bullet$] le groupe $\mathcal{D}^2(U,\mathcal{T})$ est abélien de torsion de type cofini, et pour $V \subseteq V' \subseteq U$ des ouverts non vides, on a $\mathcal{D}^2(V,\mathcal{T}) \subseteq \mathcal{D}^2(V',\mathcal{T})$. Grâce au lemme 3.7 de \cite{HS1}, on peut alors trouver un ouvert non vide $U_0$ de $U$ (dépendant de $\ell$) tel que, pour tout ouvert non vide $V$ de $U_0$, on ait $\mathcal{D}^2(V,\mathcal{T})\{\ell\} = \mathcal{D}^2(U_0,\mathcal{T})\{\ell\}$. Une chasse au diagramme montre alors que $\mathcal{D}^2(U_0,\mathcal{T})\{\ell\} = \Sha^2(K_0,T)\{\ell\}$.
\end{itemize}
 Cela achève la preuve.  
\item[(ii)] La preuve est très similaire à celle du théorème 4.4 de \cite{HS1} et à celle du théorème 3.10 de \cite{Izq1}. Il n'y a qu'une seule différence: comme $X_0$ n'est pas une variété lisse sur un corps, on ne sait pas s'il existe un accouplement $\mathbb{Z}(1) \otimes^{\mathbf{L}} \mathbb{Z}(1) \rightarrow \mathbb{Z}(2)$ dans la catégorie dérivée des faisceaux étales sur $X_0$ qui soit compatible avec le morphisme naturel $\mu_n \otimes^{\mathbf{L}} \mu_n[1] \rightarrow \mu_n^{\otimes 2}[1]$. Mais on peut voir dans la preuve du théorème 3.10 de \cite{Izq1} que l'on a en fait uniquement besoin d'un accouplement $\mathbb{Z}(1) \otimes^{\mathbf{L}} \mathbb{Q}/\mathbb{Z}(1) \rightarrow \mathbb{Q}/\mathbb{Z}(2)$. Et un tel accouplement existe bien puisque, pour chaque $n \geq 1$, d'après la proposition 1 de \cite{Kahn}, on a un morphisme dans la catégorie dérivée:
$$\mathbb{Z}(1) \otimes^{\mathbf{L}} \mathbb{Z}/n\mathbb{Z}(1) \cong (\mathbb{G}_m \otimes^{\mathbf{L}} \mathbb{G}_m)[-2] \otimes^{\mathbf{L}} \mathbb{Z}/n\mathbb{Z} \cong \mathbb{Z}/n\mathbb{Z}(2).$$
Une fois que l'on dispose de cet accouplement, on montre (ii) avec le même type d'arguments que (i).
\end{itemize}
\end{proof}

\begin{remarque}\label{prob}
\begin{itemize}
\item[$\bullet$] Dans (ii), si $k$ est un corps fini, le groupe $\Sha^2(K_0,\mathbb{G}_m)_{\text{non}-p}$ est nul et donc le groupe $\Sha^2(K_0,T')$ est fini et on a un accouplement parfait de groupes finis:
$$\Sha^1(K_0,T)_{\text{non}-p} \times \Sha^2(K_0,T')_{\text{non}-p} \rightarrow \mathbb{Q}/\mathbb{Z}.$$
Pour prouver la nullité de $\Sha^2(K_0,\mathbb{G}_m)_{\text{non}-p}$, il suffit de procéder comme dans la preuve de la proposition 3.4(2) de \cite{HS1}, en remplaçant la dualité de Lichtenbaum par le théorème 0.11 de \cite{Saito}.
\item[$\bullet$] Il est aussi possible d'obtenir un théorème pour $d$ quelconque. Pour ce faire, il faut poser $\tilde{T} = \hat{T} \otimes^{\mathbf{L}} \mathbb{Z}(d+1)$, comme dans la section 3 de \cite{Izq1}. On peut alors obtenir des accouplements parfaits de groupes finis:
\begin{gather*}
\Sha^{d+2}(K_0,\tilde{T})_{\text{non}-p} \times \overline{\Sha^2(K_0,T)}_{\text{non}-p} \rightarrow \mathbb{Q}/\mathbb{Z}\\
\Sha^1(K_0,T)_{\text{non}-p} \times \overline{\Sha^{d+3}(K_0,\tilde{T})}_{\text{non}-p} \rightarrow \mathbb{Q}/\mathbb{Z}.
\end{gather*}
La preuve est identique à celle du théorème 3.10 \cite{Izq1}.
\end{itemize}
\end{remarque}

Pour les applications arithmétiques (en particulier le théorème \ref{PLGC((x,y))}), il sera utile d'avoir aussi un théorème de dualité pour certains complexes de tores à deux termes:

\begin{theorem}\label{multC((x,y))}
Supposons $d=-1$ (par exemple $k$ séparablement clos). Soit $G$ un groupe de type multiplicatif sur $K$. Soit $\rho: T_1 \rightarrow T_2$ un morphisme de tores algébriques sur $K$ à noyau fini. Soient $G$ le complexe $[T_1 \rightarrow T_2]$ où $T_1$ est en degré -1 et $T_2$ en degré 0, et $\tilde{G}$ le complexe  $[\hat{T}_2 \rightarrow \hat{T}_1]$ où $\hat{T}_2$ est placé en degré -1 et $\hat{T}_1$ en degré 0. On a alors un accouplement parfait de groupes finis:
$$\Sha^1(K_0,G)_{\text{non}-p} \times \overline{\Sha^1(K_0,\tilde{G})}_{\text{non}-p} \rightarrow \mathbb{Q}/\mathbb{Z}.$$
\end{theorem}

\begin{proof} 
La preuve est tout à fait analogue à celle du corollaire 4.17 de \cite{Izq0}, et les idées sont similaires à celles du théorème \ref{PTC((x,t))tore}. On se contente donc d'esquisser la preuve, en supposant (pour simplifier) que $p=0$. \\Soit $U$ un ouvert non vide de $X_0$ sur lequel $T_1$ (resp. $T_2$) s'étend en un tore $\mathcal{T}_1$ (resp. $\mathcal{T}_2$) et $\rho$ s'étend en un morphisme $\mathcal{T}_1 \rightarrow \mathcal{T}_2$, aussi noté $\rho$. Soient $\mathcal{G}$ le complexe $ [\mathcal{T}_1 \rightarrow \mathcal{T}_2]$ et $\tilde{\mathcal{G}}$ le complexe $[\hat{\mathcal{T}}_2 \rightarrow \hat{\mathcal{T}}_1]$. Fixons $\ell$ un nombre premier. En utilisant le théorème \ref{AVdquelc}, on peut alors montrer qu'il y a un accouplement parfait de groupes finis:
\begin{equation} \label{sharp}
H^1(U,\mathcal{G})\{\ell\}^{(\ell)} \times H^1_c(U,\tilde{\mathcal{G}})^{(\ell)}\{\ell\} \rightarrow \mathbb{Q}/\mathbb{Z}.\tag{$\sharp$}
\end{equation}
Introduisons maintenant les groupes:
\begin{gather*}
\mathcal{D}^1(U,\tilde{\mathcal{G}}) = \text{Im} \left( H^1_c(U,\tilde{\mathcal{G}}) \rightarrow H^1(K_0,\tilde{G}) \right),\\
D_{sh}^1(U,\mathcal{G}) = \text{Ker} \left( H^1(U,\mathcal{G}) \rightarrow \prod_{v \in X_0^{(1)}} H^1(K_{0,v},G) \right).
\end{gather*}
En combinant la dualité parfaite (\ref{sharp}) avec une dualité locale, on montre que l'on a un morphisme surjectif à noyau divisible $D_{sh}^1(U,\mathcal{G})\{\ell\} \rightarrow \overline{\mathcal{D}^1(U,\tilde{\mathcal{G}})}\{\ell\}^D$. Or on peut prouver que, si $U$ est suffisamment petit, on a $\mathcal{D}^1(U,\tilde{\mathcal{G}}\{\ell\} = \Sha^1(K_0,\tilde{G})\{\ell\}$. Du coup, en passant à la limite sur $U$, on obtient un isomorphisme:
$$\overline{\Sha^1(K_0,G)\{\ell\}} \cong \overline{\Sha^1(K_0,\tilde{G})\{\ell\}}.$$
Pour conclure, il suffit de remarquer que le groupe $\Sha^1(K_0,G)\{\ell\}$ est fini (car $\rho$ est à noyau fini).
\end{proof}

\section{Applications arithmétiques}\label{secappdim2}

\hspace{4ex} On suppose toujours que $k$ est un corps $(p,d)$-bon pour la dualité et on rappelle que $R_0$ désigne une $k$-algèbre commutative, locale, géométriquement intègre (ie $R_0 \otimes_k k^s$ est intègre), normale, hensélienne, excellente, de dimension 2 et de corps résiduel $k$. On note toujours $K_0$ son corps des fractions, $\mathfrak{m}_0$ son idéal maximal, $\mathcal{X}_0$ le schéma $ \text{Spec} \; R_0$ et $X_0$ le schéma $\mathcal{X}_0 \setminus \{\mathfrak{m}_0\}$. Le but de cette section est d'utiliser les théorèmes de dualité de type Poitou-Tate de la section précédente pour étudier le principe local-global et l'approximation faible sur $K_0$.

\subsection{Principe local-global dans le cas $k=k^s$ et $\text{Car}(k)=0$}\label{k=k^s}

\hspace{4ex} Supposons que $k$ soit algébriquement clos de caractéristique 0. En particulier, $k$ est un corps $(-1,0)$-bon pour la dualité. Soient $H$ un $K_0$-groupe linéaire réductif connexe et $Y$ un espace principal homogène sous $H$ tel que, pour chaque $v \in X_0^{(1)}$, on a $Y(K_{0,v}) \neq \emptyset$. D'après la preuve de la proposition \ref{G_m}(ii), on a une suite exacte: 
\begin{equation}\label{flat}
\text{Br} K_0 \rightarrow \bigoplus_{v \in X_0^{(1)}} \text{Br} \; K_{0,v} \overset{\theta}{\rightarrow} \mathbb{Q}/\mathbb{Z} \rightarrow 0. \tag{$\flat$}
\end{equation}
En considérant un modèle géométriquement intègre $\mathcal{Y}$ de $Y$ sur un ouvert non vide $U$ de $X_0$, on peut définir un accouplement analogue à l'accouplement de Brauer-Manin pour les corps de nombres:
\begin{align*}
[.,.]:Y(\mathbb{A}_{K_0}) \times \text{Br}(Y) & \rightarrow \mathbb{Q}/\mathbb{Z}\\
\left((P_v)_{v}, \alpha \right) & \mapsto \theta \left((\alpha(P_v))_v \right), 
\end{align*}
où:
$$Y(\mathbb{A}_{K_0}) = \varinjlim_{V \subseteq U} \prod_{v \in V^{(1)}} \mathcal{Y}(\mathcal{O}_v) \times \prod_{v \in X_0\setminus V} Y(K_{0,v}).$$
Ici, $V$ décrit les ouverts non vides de $U$ et $\mathcal{O}_v$ désigne le complété de l'anneau local de $X_0$ en $v$. Notons:
$$\Brusse(Y)=\text{Ker}(\text{Br}_{\text{al}}(Y) \rightarrow \prod_{v\in X_0^{(1)}}\text{Br}_{\text{al}}(Y \times_{K_0} {K_{0,v}})).$$
Toujours en utilisant la suite exacte (\ref{flat}), on remarque que $[.,.]$ induit un accouplement:
$$[.,.]: Y(\mathbb{A}_{K_0}) \times \Brusse(Y) \rightarrow \mathbb{Q}/\mathbb{Z},$$
et que $Y(K_0)$ est contenu dans l'orthogonal de $\Brusse (Y)$ dans $Y(\mathbb{A}_{K_0})$: cela définit une obstruction au principe local-global pour $Y$, que l'on appellera \textbf{obstruction de Brauer-Manin associée à $\Brusse(Y)$}. \\

\hspace{4ex} On remarque que, pour $\alpha \in \Brusse(Y)$, la valeur de $[(P_v),\alpha]$ ne dépend pas du choix du point adélique $(P_v) \in Y(\mathbb{A}_{K_0})$. Comme on dispose d'un isomorphisme canonique $\Brusse(Y) \rightarrow \Brusse(H)$ (voir le lemme 5.2(iii) de \cite{BVH}), cela permet de définir un accouplement:
\begin{align*}
\text{BM}: \Sha^1(H) \times \Brusse(H) & \rightarrow \mathbb{Q}/\mathbb{Z}\\
([Y],\alpha) &\mapsto [(P_v),\alpha]
\end{align*}
où $(P_v) \in Y(\mathbb{A}_{K_0})$ est un point adélique quelconque.\\

\hspace{4ex} Par ailleurs, notons $H^{ss}$ le sous-groupe dérivé de $H$ et $H^{sc}$ son revêtement universel, qui est semi-simple simplement connexe. Soit $T$ un tore maximal de $H$. Notons $T^{(sc)}$ l'image réciproque de $T$ par le morphisme composé $\rho: H^{sc} \rightarrow H^{ss} \rightarrow H$ et $G = [T^{(sc)} \rightarrow T]$. Comme dans la section 2 de l'article \cite{Bor} de Borovoi, pour $L$ une extension de $K_0$, on définit la cohomologie galoisienne abélienne de $H$ par $H^r_{\text{ab}}(L,H) = H^r(L,G)$. Elle est munie de morphismes d'abélianisation $\text{ab}^1_L: H^1(L,H) \rightarrow H^1_{\text{ab}}(L,H)$ (voir section 3 de \cite{Bor}). 

\begin{proposition}\label{abbij}
Pour chaque $v \in X_0^{(1)}$, le morphisme $\text{ab}^1_{K_{0,v}}$ est bijectif. Il en est de même du morphisme $\text{ab}^1_{K_0}$.
\end{proposition}

\begin{proof}
D'après les théorèmes 1.4 et 1.5 de \cite{CGR}, $K_0$ et $K_{0,v}$ sont des corps de caractéristique 0, de dimension cohomologique 2, tels que indice et exposant coïncident pour les algèbres simples centrales, et sur lesquels la conjecture de Serre II vaut. Il suffit donc de procéder comme dans le corollaire 5.4.1 de \cite{Bor} pour obtenir l'injectivité et d'invoquer le théorème 5.1(i) et l'exemple 5.4(vi) de \cite{GA} pour obtenir la surjectivité.
\end{proof}

\hspace{4ex}  Le noyau du morphisme $T^{(sc)} \rightarrow T$ étant fini, le théorème \ref{multC((x,y))} fournit un accouplement parfait de groupes finis de type Poitou-Tate:
$$\text{PT}: \Sha^1(K_0,G) \times \overline{\Sha^1(K_0,\tilde{G})} \rightarrow \mathbb{Q}/\mathbb{Z},$$
où $\tilde{G}$ est le complexe $[\hat{T}_2 \rightarrow \hat{T}_1]$ dans lequel $\hat{T}_2$ est placé en degré -1 et $\hat{T}_1$ en degré 0. Nous disposons d'isomorphismes permettant de comparer cet accouplement à l'accouplement de Brauer-Manin:
\begin{itemize}
\item[$\bullet$] d'après la proposition \ref{abbij}, les morphismes d'abélianisation induisent une bijection $\text{ab}^1: \Sha^1(K_0,H) \rightarrow \Sha^1(K_0,G)$,
\item[$\bullet$] d'après le corollaire 2.20 et le théorème 4.8 de \cite{BVH}, les groupes $\text{Br}_{\text{al}}(H)$ et $H^1(K_0,\tilde{G})$ sont isomorphes, ce qui induit un isomorphisme $B: \Brusse(H) \rightarrow \Sha^1(K_0,\tilde{G})$.
\end{itemize}
On peut alors montrer, comme dans la proposition 2.5 de \cite{Izq3}, que l'on a un diagramme commutatif au signe près:\\
\centerline{\xymatrix{
\text{BM}: & \Sha^1(K_0,H) \ar[d]^{\text{ab}^1} \ar@{}[r]|{\times}   & \overline{\Brusse(H)}\ar[d]^{B} \ar[r] & \mathbb{Q}/\mathbb{Z}\ar@{=}[d]\\
\text{PT}: &\Sha^1(K_0,G) \ar@{}[r]|{\times}  & \overline{\Sha^1(K_0,\tilde{G})} \ar[r] & \mathbb{Q}/\mathbb{Z}.
}} 
On en déduit le théorème:

\begin{theorem}\label{PLGC((x,y))}
Supposons que $k$ soit algébriquement clos de caractéristique 0. Soit $H$ un groupe réductif connexe sur $K_0$. L'accouplement de Brauer-Manin: $$\Sha^1(K_0,H) \times \overline{\Brusse(H)} \rightarrow \mathbb{Q}/\mathbb{Z}$$
induit une bijection $\Sha^1(K_0,H) \cong \overline{\Brusse(H)}^D$. En particulier, l'obstruction de Brauer-Manin associée à $\Brusse(Y)$ est la seule obstruction pour les espaces principaux homogènes sous $H$.
\end{theorem}

\hspace{4ex} Cela répond affirmativement à la question de J.-L. Colliot-Thélène, R. Parimala et V. Suresh posée à la toute fin de l'article \cite{CTPS} dans le cas (b) lorsque $k$ est de caractéristique nulle. Plus précisément, si $Y$ est un espace principal homogène sous $H$ pour lequel il existe une famille $(P_v) \in \prod_{v \in \tilde{\mathcal{X}}_0} Y(K_{0,v})$ telle que pour tout $A \in \Brusse(Y)$:
$$(\partial_v(A(P_v)))_{v \in \tilde{\mathcal{X}}_0^{(1)}} \in \text{Ker}\left(\bigoplus_{v \in \tilde{\mathcal{X}}_0^{(1)}} H^1(k(v),\mathbb{Q}/\mathbb{Z}(1)) \overset{\partial_{v,w}}{\rightarrow} \bigoplus_{w \in \tilde{\mathcal{X}}_0^{(2)}}\mathbb{Q}/\mathbb{Z}\right),$$
alors $Y$ a un point rationnel. Ici, $ \tilde{\mathcal{X}}_0$ désigne une désingularisation de $\mathcal{X}_0$ comme dans la section \ref{desingu} et les $\partial_v$ et $\partial_{v,w}$ désignent les applications résidus.

\subsection{Principe local-global dans le cas général}

\hspace{4ex} On ne fait plus d'hypothèse sur $d$ ni sur $p$. Ainsi, le corps $k$ peut être séparablement clos de caractéristique positive, fini, $\ell$-adique, ou encore $\mathbb{C}((t))$. Soient $T$ un tore sur $K_0$ et $Y$ un espace principal homogène sous $T$ tel que $Y(\mathbb{A}_{K_0}) \neq \emptyset$ et qui devient trivial sur une extension finie de $K_0$ de degré non divisible par $p$. La théorie de Bloch-Ogus (rappelée par exemple dans le paragraphe 2 de \cite{CTPS}) fournit un complexe:
\small
\begin{equation}\label{flatflat}
H^{d+3}(K_0,\mathbb{Q}/\mathbb{Z}(d+2))_{\text{non}-p} \rightarrow \bigoplus_{v \in X_0^{(1)}} H^{d+3}(K_{0,v},\mathbb{Q}/\mathbb{Z}(d+2))_{\text{non}-p}\overset{\vartheta}{\rightarrow} H^{d+1}(k,\mathbb{Q}/\mathbb{Z}(d))_{\text{non}-p}. \tag{$\flat \flat$}
\end{equation}
\normalsize

Comme dans le paragraphe précédent, en identifiant les groupes $H^{d+1}(k,\mathbb{Q}/\mathbb{Z}(d))_{\text{non}-p}$ et $(\mathbb{Q}/\mathbb{Z})_{\text{non}-p}$, on peut alors définir un accouplement à la Brauer-Manin:
\begin{align*}
[.,.]: Y(\mathbb{A}_{K_0}) \times H^{d+3}(Y,\mathbb{Q}/\mathbb{Z}(d+2))_{\text{non}-p} &\rightarrow \mathbb{Q}/\mathbb{Z}\\
\left( (P_v)_v, \alpha \right) & \mapsto \vartheta \left( (\alpha (P_v))_v \right).
\end{align*}
Posons:
\begin{align*}
H^{d+3}_{\text{lc}}(Y,\mathbb{Q}/\mathbb{Z}(d+2)) &= \text{Ker}(H^{d+3}(Y,\mathbb{Q}/\mathbb{Z}(d+2))/\text{Im}(H^{d+3}(K_0,\mathbb{Q}/\mathbb{Z}(d+2)))\\ & \rightarrow \prod_{v \in X_0^{(1)}} H^{d+3}(Y_{K_{0,v}},\mathbb{Q}/\mathbb{Z}(d+2))/\text{Im}(H^{d+3}(K_{0,v},\mathbb{Q}/\mathbb{Z}(d+2)))
\end{align*}
où $Y_{K_{0,v}}$ désigne $Y \times_{K_0} K_{0,v}$. En utilisant de nouveau le complexe (\ref{flatflat}), on remarque que l'accouplement $[.,.]$ induit un accouplement:
$$[.,.]: Y(\mathbb{A}_{K_0}) \times H^{d+3}_{\text{lc}}(Y,\mathbb{Q}/\mathbb{Z}(d+2))_{\text{non}-p} \rightarrow \mathbb{Q}/\mathbb{Z},$$
et que $Y(K)$ est contenu dans l'orthogonal de $H^{d+3}_{\text{lc}}(Y,\mathbb{Q}/\mathbb{Z}(d+2))_{\text{non}-p}$. De plus, pour $\alpha \in H^{d+3}_{\text{lc}}(Y,\mathbb{Q}/\mathbb{Z}(d+2))_{\text{non}-p}$, la quantité $[(P_v),\alpha]$ est indépendante du choix du point adélique $(P_v) \in Y(\mathbb{A}_{K_0})$. L'accouplement $[.,.]$ induit donc un morphisme:
 $$\rho_Y: H^{d+3}_{\text{lc}}(Y,\mathbb{Q}/\mathbb{Z}(d+2))_{\text{non}-p} \rightarrow \mathbb{Q}/\mathbb{Z}.$$
 
\begin{proposition}
Soit:
$$PT: \Sha^1(K_0,T)_{\text{non}-p} \times \overline{\Sha^{d+3}(K_0,\tilde{T})}_{\text{non}-p} \rightarrow \mathbb{Q}/\mathbb{Z}$$
l'accouplement parfait de la remarque \ref{prob}. Il existe un morphisme $$\tau: \Sha^{d+3}(K_0,\tilde{T})_{\text{non}-p} \rightarrow H^{d+3}_{\text{lc}}(Y,\mathbb{Q}/\mathbb{Z}(d+2))_{\text{non}-p}$$ tel que l'on ait $\rho_Y \circ \tau = PT([Y],\cdot)$.
\end{proposition}

\begin{proof}
La preuve est analogue à la partie 5 de \cite{HS1}. On se contente de rappeller brièvement la construction de $\tau$.\\
Notons $\overline{Y} = Y \times_{K_0} K_0^s$. La suite spectrale $$H^p(K_0,H^q(\overline{Y}, \mathbb{Q}/\mathbb{Z}(d+2))) \Rightarrow H^{p+q}(Y,\mathbb{Q}/\mathbb{Z}(d+2))$$ induit un morphisme:
$$H^{d+2}(K_0,H^1(\overline{Y}, \mathbb{Q}/\mathbb{Z}(d+2))) \rightarrow H^{d+3}(Y,\mathbb{Q}/\mathbb{Z}(d+2)).$$
Comme dans le lemme 5.2 de \cite{HS1}, on a un isomorphisme de modules galoisiens $H^1(\overline{Y},\mathbb{Q}/\mathbb{Z}(d+2))_{\text{non}-p} \cong \hat{T} \otimes \mathbb{Q}/\mathbb{Z}(d+1)_{\text{non}-p}$. On obtient donc un morphisme:
$$H^{d+2}(K_0,\hat{T} \otimes \mathbb{Q}/\mathbb{Z}(d+1))_{\text{non}-p} \rightarrow H^{d+3}(Y,\mathbb{Q}/\mathbb{Z}(d+2))_{\text{non}-p}.$$
D'après le lemme 3.1 de \cite{Izq1}, les groupes $H^{d+2}(K_0,\hat{T} \otimes \mathbb{Q}(d+1))$ et $H^{d+3}(K_0,\hat{T} \otimes \mathbb{Q}(d+1))$ sont triviaux. En exploitant le triangle distingué $$\hat{T} \otimes \mathbb{Z}(d+1) \rightarrow \hat{T} \otimes \mathbb{Q}(d+1) \rightarrow \hat{T} \otimes \mathbb{Q}/\mathbb{Z}(d+1) \rightarrow \hat{T} \otimes \mathbb{Z}(d+1)[1],$$
on montre alors que $H^{d+2}(K_0,\hat{T} \otimes \mathbb{Q}/\mathbb{Z}(d+1))_{\text{non}-p} \cong H^{d+3}(K_0,\tilde{T})_{\text{non}-p}$, ce qui fournit un morphisme:
$$ H^{d+3}(K_0,\tilde{T})_{\text{non}-p} \rightarrow H^{d+3}(Y,\mathbb{Q}/\mathbb{Z}(d+2))_{\text{non}-p}.$$
En passant aux éléments localement triviaux, on obtient un morphisme:
$$\Sha^{d+3}(K_0,\tilde{T})_{\text{non}-p} \rightarrow H^{d+3}_{\text{lc}}(Y,\mathbb{Q}/\mathbb{Z}(d+2))_{\text{non}-p}.$$
C'est le morphisme $\tau$.
\end{proof}
 
La proposition précédente entraîne alors le théorème suivant:

\begin{theorem}\label{local-global}
Soient $T$ un tore sur $K_0$ et $Y$ un espace principal homogène sous $T$ tel que $Y(\mathbb{A}_{K_0}) \neq \emptyset$ et qui devient trivial sur une extension finie de $K_0$ de degré non divisible par $p$. Si $\rho_Y$ est trivial, alors $Y(K_0) \neq \emptyset$.
\end{theorem}

\hspace{4ex} Comme dans le paragraphe \ref{k=k^s}, en prenant $d=-1$, cela répond affirmativement à la question de J.-L. Colliot-Thélène, R. Parimala et V. Suresh posée à la toute fin de l'article \cite{CTPS} dans le cas (b) lorsque $k$ est de caractéristique quelconque.\\

\hspace{4ex} Pour terminer ce paragraphe, remarquons qu'en utilisant la remarque \ref{prob} et en procédant comme dans le corollaire 5.7 de \cite{HS1}, on peut obtenir la proposition:

\begin{proposition}\label{stabrat}
Supposons $k$ fini de caractéristique $p$ et considérons un tore stablement rationnel $T$ sur $K_0$. Soit $Y$ un espace principal homogène sous $T$ tel que $Y(\mathbb{A}_{K_0}) \neq \emptyset$ et qui devient trivial sur une extension finie de $K_0$ de degré non divisible par $p$. Alors $Y$ vérifie le principe local-global.
\end{proposition}

\begin{proof}
Comme $T$ est stablement rationnel, il existe une résolution de $T$:
$$0 \rightarrow T' \rightarrow R \rightarrow S \rightarrow 0$$
où $R$ et $S$ sont des tores quasi-triviaux (proposition 6 de \cite{CTS}). Or la remarque \ref{prob} impose que $\Sha^2(K_0,R)_{\text{non}-p}$ est trivial. On en déduit que $\Sha^2(K_0,T')_{\text{non}-p} = 0$, et le théorème \ref{PTC((x,t))tore} permet de conclure que $\Sha^1(K_0,T)_{\text{non}-p}=0$.
\end{proof}

\subsection{Approximation faible}

\hspace{4ex} On suppose dans ce paragraphe que $p=0$. En particulier, $k$ est de caractéristique nulle. Pour $M$ un $\text{Gal}(K_0^s/K_0)$-module discret, on note $\Sha^{2}_{\omega}(K_0,M)$ l'ensemble des éléments de $H^2(K_0,M)$ dont la restriction à $H^2(K_{0,v}, M)$ est triviale pour presque tout $v \in X_0^{(1)}$. On rappelle la définition 1.1 de \cite{Izq3}:

\begin{definition} \label{def}
Soit $T$ un $K_0$-tore. On dit que $T$ vérifie \textbf{l'approximation faible} si $T(K_0)$ est dense dans $\prod_{v \in X_0^{(1)}} T(K_{0,v})$, où le groupe $T(K_{0,v})$ est muni de la topologie $v$-adique pour chaque $v$. On dit que $T$ vérifie \textbf{l'approximation faible faible} (resp. \textbf{l'approximation faible faible dénombrable}) s'il existe une partie finie (resp. dénombrable) $S_0$ de $X_0^{(1)}$ telle que, pour toute partie finie $S$ de $X_0^{(1)}$ n'intersectant pas $S_0$, le groupe $T(K_0)$ est dense dans $\prod_{v \in S} T(K_{0,v})$. 
\end{definition}

 \hspace{4ex} Exactement comme dans la section 9 de \cite{CTH}, dans la section 3 de \cite{HS2} ou encore dans la section 1 de \cite{Izq3}, on peut comprendre l'obstruction à l'approximation faible:

\begin{theorem} \label{AFC((x,y))}
Soit $T$ un $K_0$-tore de module de caractères $\hat{T}$ et de tore dual $T'$. On note $ \overline{T(K_0)}$ l'adhérence de $T(K_0)$ dans $\prod_{v \in X_0^{(1)}} T(K_{0,v})$. 
\begin{itemize}
\item[(i)] Supposons $d=-1$ (par exemple $k=\mathbb{C}$). On a la suite exacte: $$0 \rightarrow \overline{T(K_0)} \rightarrow \prod_{v \in X_0^{(1)}} T(K_{0,v}) \rightarrow (\Sha^{2}_{\omega}(K_0,\hat{T}))^D \rightarrow (\Sha^{2}(K_0,\hat{T}))^D \rightarrow 0.$$
Le tore $T$ vérifie toujours l'approximation faible faible, et il vérifie l'approximation faible si, et seulement si, $\Sha^{2}(K_0,\hat{T}) = \Sha^{2}_{\omega}(K_0,\hat{T})$.
\item[(ii)] Supposons que $d=0$ (par exemple $k = \mathbb{C}((t))$). On a la suite exacte:
$$0 \rightarrow \overline{T(K_0)} \rightarrow \prod_{v \in X_0^{(1)}} T(K_{0,v}) \rightarrow (\Sha^{2}_{\omega}(K_0,T'))^D \rightarrow (\Sha^{2}(K_0,T'))^D \rightarrow 0.$$
Le tore $T$ vérifie l'approximation faible si, et seulement si, $\Sha^{2}(K_0,T') = \Sha^{2}_{\omega}(K_0,T')$. Le tore $T$ vérifie l'approximation faible faible si, et seulement si, le groupe de torsion $\Sha^{2}_{\omega}(K_0,T')$ est de type cofini.  Le tore $T$ vérifie l'approximation faible faible dénombrable si, et seulement si, le groupe de torsion $\Sha^{2}_{\omega}(K_0,T')$ est dénombrable.
\end{itemize}
\end{theorem}

\begin{proof} \textit{(Esquisse).}
Voici les grandes étapes pour (i):
\begin{itemize}
\item[1.] Pour $F$ un $\text{Gal}(K_0^s/K_0)$-module fini, on établit une suite exacte:
$$H^{1}(K_0,F') \rightarrow \mathbb{P}^1(K_0,F') \rightarrow H^1(K_0,F)^D$$
où $F' = \text{Hom}(F,\mu_n)$ et $\mathbb{P}^1(K_0,F')$ est le produit restreint des $H^1(K_{0,v},F')$ pour $v \in X^{(1)}$ par rapport aux $H^1(\mathcal{O}_v,F')$.
\item[2.] En dualisant la suite exacte précédente, on montre que, pour $F$ un $\text{Gal}(K_0^s/K_0)$-module fini et $S$ une partie finie de $X_0^{(1)}$, on a une suite exacte:
$$H^1(K_0,F) \rightarrow \prod_{v \in S} H^1(K_{0,v},F) \rightarrow \Sha^1_S(K_0,F')^D \rightarrow \Sha^1(K_0,F')^D \rightarrow 0,$$
où $\Sha^1_S(K_0,M) = \text{Ker}\left( H^1(K_0,M) \rightarrow \prod_{v \in X_0^{(1)} \setminus S} H^1(K_{0,v},M) \right)$ pour chaque module galoisien $M$.
\item[3.] L'exploitation de la suite exacte précédente avec $F=\mu_n$ permet de montrer que $\Sha^2_S(K_0,\mathbb{Z}) = \Sha^2(K_0,\mathbb{Z})$. D'après le théorème \ref{PTC((x,t))tore}, c'est un groupe de torsion de type cofini divisible. De plus, l'égalité $\Sha^2_S(K_0,\mathbb{Z}) = \Sha^2(K_0,\mathbb{Z})$ permet de vérifier que le théorème est vrai pour les tores quasi-triviaux.
\item[4.] Comme le théorème est vrai pour les tores quasi-triviaux, le lemme d'Ono (théorème 1.5.1 de \cite{Ono}) permet de supposer que $T$ admet une résolution:
$$0 \rightarrow F \rightarrow R \rightarrow T \rightarrow 0$$
avec $F$ fini et $R$ tore quasi-trivial. Cela donne un diagramme commutatif:\\
\xymatrix{
& R(K_0) \ar[r] \ar[d] & T(K_0) \ar[r]\ar[d] & H^1(K_0,F) \ar[d]\ar[r] & 0\\
& \prod_{v \in S} R(K_{0,v}) \ar[r]\ar[d] & \prod_{v \in S} T(K_{0,v}) \ar[r]\ar[d]  & \prod_{v \in S} H^1(K_{0,v},F) \ar[d]\ar[r] & 0 \\
0 \ar[r]& \Sha^2_S(K_0,\hat{R})^D \ar[r] & \Sha^2_S(K_0,\hat{T})\ar[r] & \Sha^1_S(K_0,F')\ar[r] & 0.
}\\
Les deux premières lignes sont bien sûr exactes, et on peut montrer que la troisième l'est aussi grâce à l'étape 3. En utilisant encore l'étape 3, une chasse au diagramme permet d'établir une suite exacte:
$$0 \rightarrow \overline{T(K_0)_S} \rightarrow \prod_{v \in S} T(K_{0,v}) \rightarrow (\Sha^2_S(K_0,\hat{T}))^D,$$
où $\overline{T(K_0)_S}$ désigne l'adhérence de $T(K_0)$ dans $\prod_{v \in S} T(K_{0,v})$.
En passant à la limite projective sur $S$, on obtient l'exactitude de:
 $$0 \rightarrow \overline{T(K_0)} \rightarrow \prod_{v \in X_0^{(1)}} T(K_{0,v}) \rightarrow (\Sha^{2}_{\omega}(K_0,\hat{T}))^D .$$
 \item[5.] L'exactitude de $$\prod_{v \in X_0^{(1)}} T(K_{0,v}) \rightarrow (\Sha^{2}_{\omega}(K_0,\hat{T}))^D \rightarrow (\Sha^{2}(K_0,\hat{T}))^D \rightarrow 0$$
 s'obtient en dualisant la suite exacte:
 $$0 \rightarrow \Sha^2(K_0,\hat{T}) \rightarrow \Sha^2_{\omega}(K_0,\hat{T}) \rightarrow \bigoplus_{v \in X_0^{(1)}} H^2(K_{0,v},\hat{T}).$$
 \item[6.] De la suite exacte:
 $$0 \rightarrow \overline{T(K_0)} \rightarrow \prod_{v \in X_0^{(1)}} T(K_{0,v}) \rightarrow (\Sha^{2}_{\omega}(K_0,\hat{T}))^D \rightarrow (\Sha^{2}(K_0,\hat{T}))^D \rightarrow 0$$
 on déduit que $T$ vérifie l'approximation faible si, et seulement si, $\Sha^{2}(K_0,\hat{T}) = \Sha^{2}_{\omega}(K_0,\hat{T})$. et qu'il vérifie l'approximation faible faible si, et seulement si, $\Sha^{2}(K_0,\hat{T})$ est d'indice fini dans $\Sha^{2}_{\omega}(K_0,\hat{T})$. En utilisant une résolution flasque de $T$, on peut montrer que $\Sha^{2}_{\omega}(K_0,\hat{T}) = \Sha^{2}_{S_T}(K_0,\hat{T})$, où $S_T$ désigne l'ensemble des places $v \in X_0^{(1)}$ de mauvaise réduction de $T$. Ainsi, $\Sha^{2}(K_0,\hat{T})$ est toujours d'indice fini dans $\Sha^{2}_{\omega}(K_0,\hat{T})$ et $T$ vérifie l'approximation faible faible.
\end{itemize}
\end{proof}

\begin{remarque}
\begin{itemize}
\item[$\bullet$] La preuve de (ii) est similaire.
\item[$\bullet$] Comme dans la remarque \ref{prob}, on pourrait obtenir un résultat pour l'approximation faible dans le cas où $d$ est quelconque en faisant intervenir les complexes de Bloch.
\end{itemize}
\end{remarque}

\end{document}